\documentclass{article}

\usepackage{amsmath,amsthm,amssymb,tikz,bbm,float}
\usepackage{empheq}
\usepackage[nottoc,numbib]{tocbibind}

\newtheorem{theorem}{Theorem}[section]
\newtheorem{corollary}[theorem]{Corollary}
\newtheorem{lemma}[theorem]{Lemma}
\newtheorem{proposition}[theorem]{Proposition}

\theoremstyle{definition}
\newtheorem{definition}[theorem]{Definition}

\newtheorem{example}[theorem]{Example}
\newtheorem{note}[theorem]{Note}

\usepackage[bottom]{footmisc}

\title{Near-integral fusion}

\author{\small{Jingcheng Dong and Andrew Schopieray}}

\date{}

\begin{document}

\maketitle

\begin{abstract}
We abstract the study of irreducible characters of finite groups vanishing on all but two conjugacy classes, initiated by S.\ Gagola, to irreducible characters of fusion rings whose kernel has maximal rank.  These near-integral fusion rings include the near-groups which are currently one of the most abundant sources of novel examples of fusion categories to date.  We generalize many of the known results on near-group fusion categories from the literature to near-integral fusion categories and characterize when such categories are braided.  In particular, braided near-integral fusion categories describe all braided fusion categories which are almost symmetrically braided.  This novel result allows a digestible characterization of the over $300$ braided equivalence classes of premodular fusion categories of rank $6$ or less.
\end{abstract}

\tableofcontents


\section{Introduction}

\par The irreducible characters of finite groups vanishing on all but two conjugacy classes were initially studied by S.\ Gagola in the 1980's \cite{MR721927} and have since been coined \emph{Gagola characters}.  A Gagola character $\chi$ of a finite group $G$ is unique when it exists and $|G|>2$.  One observes that the degree of this irreducible character is large compared to $|G|$ which naturally leads to a study of pairs $\chi,G$ of finite groups with an irreducible character $\chi$ such that $|G|/\mathrm{deg}(\chi)=|G|+k$ for small integers $k\in\mathbb{Z}_{\geq0}$.  Only the trivial group has $\chi$ such that $k=0$ and those $G$ with $\chi$ such that $k=1$ are either order $2$ or in an infinite family of doubly-transitive Frobenius groups.  Interestingly, for $k>1$, there are finitely many such pairs $G,\chi$ with an initial bound given by N.\ Snyder in \cite{MR2383494} which has since been made sharp \cite{MR3276228} (see also Example \ref{ex:exspecial}).

\par From a modern outlook, the characters of a finite group $G$ form an algebraic structure known as a fusion ring and one could equivalently attempt to describe Gagola characters of these more general objects.  One quickly finds the literal generalization, of irreducible characters vanishing on all but two conjugacy classes, bears little resemblance at all to the original theory but may be interesting in its own right.  We instead seek in this manuscript characters of fusion rings with a proper kernel \cite[Definition 3.1]{MR3552799} of maximal rank, and prove (Theorem \ref{thm:char}) that a maximal fusion subring $S\subset R$ exists with $\mathrm{rank}(S)+1=\mathrm{rank}(R)$ if and only if $S$ is the kernel of a distinguished irreducible character of $R$.  As a corollary of this characterization and \cite[Theorem 5.7]{MR4413278}, a finite group has a Gagola character if and only if the character ring of $G$ has an irreducible character whose kernel is a proper fusion subring of maximal rank, supporting the fact that this is the natural generalization of Gagola characters for fusion rings.

\par Meanwhile, the study of representation theory of finite groups $G$ has inspired the study of more general categorical structures known as fusion categories \cite{ENO}, with the categories of $G$-graded vector spaces and the finite-dimensional representation theory of $G$ being canonical examples with other large families of categories coming from the representation theory of quantum groups (see \cite{MR4079742} and references within).  Applications outside of representation theory include areas of mathematical physics (see \cite{MR3308880} and references within) and quantum computation \cite{MR3777017} utilizing braided fusion categories as models of anyonic topological orders.  But one major open problem in the field of fusion categories is the construction of examples.  There are very few examples of fusion categories known which are not in some way related to either the representation theory of quasi-Hopf algebras or quantum groups through a small list of categorical constructions.  One family of exotic examples in this sense are the near-group fusion categories (see \cite{MR3167494} and references within).  The near-group fusion categories are fusion categories with exactly one noninvertible isomorphism class of simple objects and have provided the most promise for creating an infinite family of examples which cannot be described by classical representation theory and basic categorical constructions, though only finitely many are known to exist at this time.

\par Inspired by the near-group fusion categories, we label the fusion categories possessing a proper fusion category of maximal rank \emph{near-integral}, as the Frobenius-Perron dimensions of objects in the maximal subcategory must be integers.  These categories were first studied explicitly in \cite{MR4413278} by G.\ Chen, Z.\ Wang, and the first author.  In this manuscript we prove that many of interesting results about near-group fusion categories generalize to near-integral fusion categories including the fact that they are Galois conjugate to pseudounitary fusion categories (Corollary \ref{cor:1}), they have heavily restricted fusion rules (Lemma \ref{lem:divides}), and they are commutative when the Frobenius-Perron dimension is irrational (Proposition \ref{propcomm}).  It is evident that near-integral fusion categories are intimately related with fusion categories containing a simple object of large dimension, as is the case with finite groups; we leave this line of investigation to future work.

\par Near-integral fusion categories also appear naturally when studying braided fusion categories.  Symmetrically braided fusion categories are in correspondence with the representation categories of finite groups up to small alterations of the natural braiding (see Example \ref{exsynm}).  Even when a braided category is not symmetrically braided, it contains a symmetrically braided subcategory known as the symmetric center.  Hence the study of braided near-integral fusion categories includes all braided fusion categories which are almost symmetrically braided in this sense.  These almost symmetrically braided fusion categories can be seen as a mirror of the slightly degenerate braided fusion categories \cite[Definition 2.1]{MR3022755}, which includes the important class of supermodular fusion categories in the presence of a spherical structure.  In Section \ref{sec:brad}, we characterize nonsymmetrically braided near-integral fusion categories as equivariantizations of pointed modular fusion categories of ranks $2$, $3$, or $4$ (Propositions \ref{prop:tan} and \ref{prop:stan}) with exactly four exceptions (Proposition \ref{prop:fib}) which are not integral.  This generalizes a similar classification obtained by J.\ Thornton in the case of braided near-groups \cite{thornton2012generalized} but only finitely many examples exist in this stricter setting aside from the classical examples of D.\ Tambara and S.\ Yamagami \cite{MR1659954} (see also \cite{siehler}).

\par As an application of our classification, we are able to explicitly describe the braided equivalence classes of premodular fusion categories, i.e.\ braided fusion categories equipped with a spherical structure, of rank $6$ and less in Section \ref{sec:class}.   The assumption of a spherical structure is very mild.  At the present time it is not known if all fusion categories possess a spherical structure but all known examples do.  The history of the classification of premodular fusion categories appears brief, but a wealth of studies complementary to this subject have been vital to its progression and are cited in Section \ref{sec:class} as needed.  We summarize here only the papers with explicit classification of premodular categories as a goal.  All fusion categories of ranks $1$ and $2$ have premodular structures.  The former are trivial (see \cite[Corollary 4.4.2]{tcat}, for example) and all rank $2$ fusion categories were classified in \cite{ostrik} by V.\ Ostrik who then classified premodular fusion categories of rank $3$ in \cite{premodular}; the braided assumption was then latter removed in \cite{ost15}.  Subsequently, and following the classification of modular data of ranks $4$ \cite{MR2544735} and $5$ \cite{MR3632091}, P.\ Bruillard described the possible fusion rules of premodular fusion categories of rank $4$ \cite{MR3548123} and rank $5$ \cite{MR3743161} with C.\ M.\ Ortiz-Marrero, but the braided equivalence classes of categories were not described.  We fill in this missing piece from the literature and extend the classification to rank $6$.  Aside from the clear benefit of having a database of small rank braided fusion categories, an explicit description of braided equivalence classes of premodular fusion categories up to rank $6$ allows a glimpse of the growth rate of the number of such categories up to equivalence.  It is also interesting to note how the proportion of equivalence classes of categories from symmetrically braided to nondegenerately braided changes as rank increases.

\par The $29$ braided equivalence classes of premodular fusion categories of ranks $\leq3$ are reiterated in Figure \ref{fig:catranklessthan4}.  There is a unique class of rank $1$, $8$ of rank $2$, and $20$ of rank $3$.  The $57$ braided equivalence classes of rank $4$ premodular categories are cataloged in Figure \ref{fig:catrank4}.  The $55$ braided equivalence classes of rank $5$ premodular categories are cataloged in Figures \ref{fig:catrank5} and \ref{fig:catrank5b}.  The $224$ braided equivalence classes of rank $6$ premodular categories are cataloged in Figures \ref{fig:catrank56}--\ref{fig:rank34567}, \ref{fig:rank23456}, and \ref{fig:nondegen}.  With few exceptions, all relevant categories are described with the notation in Figure \ref{fig:notation}.

\begin{figure}[H]
\centering
\begin{align*}
\begin{array}{|c|c|c|}
\hline\text{Notation} & \text{Input data} & \text{Description/reference} \\\hline\hline
\mathrm{Rep}(G,\nu) & \text{Finite group }G\text{ and central} & \mathrm{Rep}(G)\text{ with symmetric braiding} \\
 & \nu\in G\text{ with }\nu^2=e & (\text{see Section 9.9 of }\cite{tcat}) \\\hline
\mathrm{Vec} & \text{none} & \mathrm{Rep}(C_1) \\\hline
\mathrm{sVec} &\text{none} & \mathrm{Rep}(C_2,\nu),\,\,\nu\text{ nontrivial} \\\hline
\mathcal{I}_q & \text{primitive 16th root of unity }q & \text{Ising braided fusion categories} \\\hline
\mathrm{Rep}(G)^\alpha & \text{Finite group }G\text{ and }& \mathrm{Rep}(G)\text{ with nonsymmetric}\\
 &\text{root of unity }\alpha & \text{braiding }(\text{see Example }\ref{braidfin}) \\\hline
\mathcal{C}(G,q) & \text{Abelian group }G\text{ and} & \text{Pre-metric group categories} \\
 & \text{quadratic form }q:G\to\mathbb{C}^\times & (\text{see Example }\ref{ex:braidgroup}) \\\hline
\mathcal{C}(X,\ell,q) & \text{Dynkin label }X,\text{ integer }\ell  & \text{Premodular quantum group} \\
 & \text{and }q^2\text{ an }\ell^{\tiny\mathrm{th}}\text{ root of unity} & \text{categories }(\text{see Example }\ref{exquant}) \\\hline
\end{array}
\end{align*}
    \caption{Fixed notation for frequently-mentioned braided categories}%
    \label{fig:notation}%
\end{figure}


\section*{Acknowledgement}

The first author was funded in part by the Natural Science Foundation of Jiangsu Providence (Grant No. BK20201390).  The authors would like to thank Gert Vercleyen, S\'ebastien Palcoux and Dmitri Nikshych for fruitful discussion during the finalization of this manuscript.  The second author would like to thank Jack Pech for assistance during the preparation of the manuscript.


\section{Preliminaries}\label{sec:prem}

The definitions and concepts presented in this section are mostly standard.  We refer the interested reader to \cite{tcat} for further details.  Many examples are included to illustrate concepts that will also be necessary for arguments in Section \ref{sec:class}.

\subsection{Fusion rings}\label{ssec:ring}

\par A fusion ring is an abstracted version of the integral group ring $\mathbb{Z}G$ of a finite group $G$.  Roughly, it is a ring with a finite additive basis whose \emph{fusion rules}, or structure constants in the chosen basis, are nonnegative integers but there is a looser notion of inverse: the involution known as \emph{duality}.

\begin{definition}[{\cite[Chapter 3]{tcat}}]
A pair $(R,B)$ of a unital associative ring $R$ which is free as a $\mathbb{Z}$-module and a basis $B=\{1_R=:b_1,\ldots,b_n\}$ is a \emph{fusion ring} if $b_ib_j=\sum_{k=1}^nc_{ij}^kb_k$ with $c_{ij}^k\in\mathbb{Z}_{\geq0}$, and there exists an involution $\ast$ of $\{1,\ldots,n\}$ such that the map $a=\sum_{j=1}^na_jb_j\mapsto a^\ast=\sum_{j=1}^na_jb_{j^\ast}$ is an antiautomorphism of $R$ with the property that $c_{ij}^1=1$ if $i=j^\ast$ and $c_{ij}^1=0$ otherwise.
\end{definition}

One can study the (complex) representation theory of a fusion ring $(R,B)$ in a manner analogous that of finite groups as the complexified ring $R\otimes_\mathbb{Z}\mathbb{C}$ is semisimple \cite[1.2(a)]{MR933415}.  As expected, there is a well-developed character theory for fusion rings, and even more general combinatorial objects \cite[Section 2]{MR2535395}, including common results such as orthogonality of characters \cite[Propositions 2.12--2.13]{MR2535395}.  A striking difference between integral group rings $\mathbb{Z}G$ and fusion rings as a whole is the lack of a trivial representation/character $\iota:R\to\mathbb{C}$, i.e.\ mapping all basis elements to $1$.  The trivial representation is instead replaced by the \emph{Frobenius-Perron} representation $\mathrm{FPdim}:R\to\mathbb{C}$ which maps every basis element $b_i\in B$ to the maximal real eigenvalue of the fusion matrix $N_i:=[c_{ij}^k]_{j,k}$ \cite[Section 3.3]{tcat}.  As shown in \cite[Proposition 1.6]{2019arXiv191212260G}, given a fusion ring $(R,B)$ and a number field $\mathbb{K}$, the elements $x\in R$ such that $\mathrm{FPdim}(x)\in\mathbb{K}$ form a fusion subring $R_\mathbb{K}\subseteq R$.  Common uses for such a result include that every fusion ring has a maximal integral subring, $R_\mathbb{Q}$, and a maximal weakly integral subring $R_\mathbb{K}$ where $\mathbb{K}:=\mathbb{Q}(\sqrt{n}:n\in\mathbb{Z}_{\geq2})$.

\begin{definition}
Let $(R,B)$ be a fusion ring.  We say $x\in R$ is \emph{invertible} if $\mathrm{FPdim}(x)=1$.  The fusion subring generated by all invertible elements of $R$ is denoted $R_\mathrm{pt}$ and if $R=R_\mathrm{pt}$ we say $R$ is \emph{pointed}; in this case $R\cong\mathbb{Z}G$ for a finite group $G:=B$ with duality $g^\ast=g^{-1}$.
\end{definition}

Let $B=\sqcup_jB_j$ be a partition of $B$.  Denote by $R_j$, the subset of $R$ additively generated by all elements of $B_j$.  We say that a fusion ring $(R,B)$ is \emph{graded} by a finite group $G$, or is \emph{$G$-graded}, if there is a partition $B=\sqcup_{g\in G}B_g$ such that each $B_g$ is nonempty and $xy\in B_{gh}$ for any $x\in B_g$ and $y\in B_h$.  Each fusion ring $(R,B)$ has a \emph{universal grading group} $U(R)$ such that if $R$ is $H$-graded for any other finite group $H$, then there exists a surjective group homomorphism $\pi:G\to H$ \cite[Corollary 3.6.6]{tcat}.  The trivial component of the universal grading is denoted $R_\mathrm{ad}$, the \emph{adjoint subring} of $R$.  It is well-known that this subring is multiplicatively generated by $xx^\ast$ over all $x\in B$ \cite[Section 3.6]{tcat}.


\subsection{Fusion categories}

A (skeletal) \emph{fusion category} is roughly a fusion ring along with solutions to the pentagon equations \cite[Equation 2.2]{tcat} which describe possible associative structures for the categorified monoidal product.  In particular, solutions to the pentagon equations for a fusion ring $(R,B)$ prove the existence of a skeletal semisimple rigid monoidal category $\mathcal{C}$ with finitely many simple objects corresponding to the elements of $B$ whose Grothendieck ring is $R$.  In this case we say $\mathcal{C}$ is a \emph{categorification} of $(R,B)$ (see \cite[Section 4.10]{tcat}).  Any category is equivalent to a skeletal one with the axiom of choice, so it is possible to treat the combinatorial/algebraic definition of a fusion category on equal footing with the poignantly categorical definition \cite[Definition 4.1.1]{tcat}.  All the notation, definitions, and concepts of fusion rings from Section \ref{sec:prem} can then be applied to any categorification of a fusion ring.

\begin{example}\label{ex:p}
The simplest examples are the \emph{pointed fusion categories}, i.e.\ those fusion categories $\mathcal{C}$ with $\mathcal{C}=\mathcal{C}_\mathrm{pt}$.  The Grothendieck ring of such a category $\mathcal{C}$ is $\mathbb{Z}G$ for a finite group $G$.  Any $\omega:G\times G\times G\to\mathbb{C}$ which satisfies the pentagon equations lies in $\omega\in H^3(G,\mathbb{C}^\times)$ and vice-versa, and thus any pointed fusion category, or equivalently any categorification of $\mathbb{Z}G$ for a finite group $G$, is equivalent to $\mathrm{Vec}_G^\omega$, the category of $G$-graded vector spaces with associativity defined by some $\omega\in H^3(G,\mathbb{C}^\times)$.  This pairing is not a one-to-one correspondence.  Indeed, tensor equivalence classes of $\mathrm{Vec}_G^\omega$ are parameterized by orbits $H^3(G,\mathbb{C}^\times)/\mathrm{Out}(G)$ where $\mathrm{Out}(G)$ is the set of outer automorphisms of $G$ \cite[Proposition 2.6.1]{tcat}.
\end{example}

\par With a robust set of examples of fusion categories coming from finite groups and representation theory, one can create even more with basic constructions such as products \cite[Section 4.6]{tcat} denoted by $\boxtimes$, graded extensions by finite groups \cite{MR2677836}, and Galois conjugacy \cite{davidovich2013arithmetic}.

\par One additional categorical structure that will play a lesser role in this manuscipt is that of a \emph{pivotal structure} which defines a notion of categorical trace of endomorphisms in a fusion category \cite[Section 4.7]{tcat}.  If the categorical traces defined by a pivotal structure $\varphi$ on a fusion category $\mathcal{C}$ of the identity map on all simple objects $X$ and $X^\ast$ in $\mathcal{C}$ are equal, we denote these $\dim(X)$ and say $\varphi$ is a \emph{spherical structure} and that $\mathcal{C}$, equipped with $\varphi$, is a \emph{spherical fusion category}.  It is currently an open problem to determine whether or not every fusion category possesses a spherical structure, although it is known that the class of \emph{pseudounitary} fusion categories possess a canonical spherical structure \cite[Proposition 9.5.1]{tcat} such that $\dim(X)=\mathrm{FPdim}(X)$ for all objects $X$.  All integral and weakly integral fusion categories are pseudounitary, and we prove in Lemma \ref{cor:1} that every near-integral fusion category has a canonical spherical structure as well, being Galois conjugate to a pseudounitary fusion category.  The benefit of a spherical structure/pseudounitarity is access to a wider array of number-theoretical tools when studying fusion categories.  It is important to note that there exist fusion rings of very small rank which possess only non-pseudounitary categorifications \cite{MR4327964}.


\subsection{Braided fusion categories}\label{subsec:braid}

A braiding on a fusion category $\mathcal{C}$ is a collection of natural isomorphisms $\sigma:X\otimes Y\to Y\otimes X$ for all $X,Y\in\mathcal{O}(\mathcal{C})$ satisying the hexagon equations \cite[Definition 8.1.1]{tcat}; in particular, if $\sigma_{X,Y}$ is illustrated as crossing strands then the hexagon equations state that morphisms involving the $\sigma$'s satisfy the Reidemeister III relation.  Tensor functors between braided fusion categories are braided if they respect both braidings in a natural way \cite[Definition 8.1.7]{tcat} but unlike braidings on categories, this is a property of a tensor functor and not an additional structure.  In what follows it will be indicated clearly whether we are referring to categories being equivalent as fusion categories or as braided fusion categories as braidings on fusion categories are in general not unique.  For example, if $\{\sigma_{X,Y}\}_{X,Y}$ describes a braiding on $\mathcal{C}$, then the \emph{reverse} braided category $\mathcal{C}^\mathrm{rev}$ is a braided fusion category with underlying fusion category $\mathcal{C}$ and braiding $\{\sigma_{Y,X}^{-1}\}_{X,Y}$ which may or may not be braided equivalent to $\mathcal{C}$.

\begin{example}\label{ex:braidgroup}
A quadratic form on a finite abelian group $G$ is a map $q:G\to\mathbb{C}^\times$ such that $q(g)=q(g^{-1})$ and the symmetric function $b(g,h):=q(gh)q(g)^{-1}q(h)^{-1}$ is a bicharacter of $G$.  The pairing of $G$ and a quadratic form $q$ is known as a \emph{pre-metric group}.  One may verify that given a braiding $\{\sigma_{g,h}\}_{g,h\in G}$ on a pointed braided fusion category $\mathrm{Vec}_G^\omega$, the map $q:G\to\mathbb{C}^\times$ given by $q(g)=\sigma_{g,g}\in\mathrm{Aut}(g^2)\cong\mathbb{C}^\times$ is a quadratic form.  Furthermore, there is a unique braided fusion category, up to braided equivalence, corresponding to each pre-metric group in this way \cite[Theorem 8.4.12]{tcat}.  As such, we denote pointed braided fusion categories by $\mathcal{C}(G,q)$ for a finite group $G$ and quadratic form $q:G\to\mathbb{C}^\times$.

\par For an elementary demonstration of this, consider braidings on the pointed fusion categories $\mathrm{Vec}_{C_p}^\omega$ where $p\in\mathbb{Z}_{\geq3}$ is an odd prime and $\omega\in H^3(C_p,\mathbb{C}^\times)$.  By the classification of pre-metric groups, there are exactly $3$ braided fusion categories $\mathcal{C}(C_p,q)$ up to braided equivalence; the distinct quadratic forms are either trivial, or take values in $\zeta_p^a$ where $a$ is either a quadratic residue or quadratic non-residue modulo $p$ and $\zeta_p=\exp(2\pi i/p)$.
It is tempting to assume these correspond to various equivalence classes of $\mathrm{Vec}_{C_p}^\omega$, which is misguided.   Specifically, $\mathrm{Vec}_{C_p}$ (trivial $\omega$) has all three inequivalent braidings corresponding to $\mathcal{C}(C_p,q)$ while $\mathrm{Vec}_{C_p}^\omega$ for nontrivial $\omega$ have none.  In fact, $\mathrm{Vec}_G^\omega$ for a group of odd order is braided if and only if $\omega$ is trivial.   Compare this with $p=2$.  There are exactly $4$ braided fusion categories $\mathcal{C}(C_2,q)$ up to equivalence; the distinct quadratic forms take a nontrivial value which is a fourth root of unity.   Underlying these braided fusion categories are $\mathrm{Vec}_{C_2}$ with exactly two inequivalent braidings $\mathcal{C}(C_2,q)$ with $q$ taking values in the second roots of unity and $\mathrm{Vec}_{C_2}^\omega$ for nontrivial $\omega$ with exactly two inequivalent braidings $\mathcal{C}(C_2,q)$ with $q$ taking values in the primitive fourth roots of unity.
\end{example}

\par Two objects $X$ and $Y$ in a braided fusion category $\mathcal{C}$ are said to \emph{centralize} one another if $\sigma_{Y,X}\sigma_{X,Y}=\mathrm{id}_{X,Y}$.  If $\mathcal{D}\subset\mathcal{C}$ is a fusion subcategory of a braided fusion category, the subset $C_\mathcal{C}(\mathcal{D})$ of $X\in\mathcal{C}$ such that $X$ and $Y$ centralize one another for all $Y\in\mathcal{D}$ forms a full fusion subcategory of $\mathcal{C}$ called the \emph{centralizer} of $\mathcal{D}$ in $\mathcal{C}$.  An extreme version of this is the \emph{symmetric center} $C_\mathcal{C}(\mathcal{C})$ of a braided fusion category $\mathcal{C}$.  When $C_\mathcal{C}(\mathcal{C})$ is trivial, we say $\mathcal{C}$ is \emph{nondegenerately braided}.

\begin{example}\label{exsynm}
Braided fusion categories such that $C_\mathcal{C}(\mathcal{C})=\mathcal{C}$ are known as \emph{symmetrically braided fusion categories} or often \emph{symmetric fusion categories}.  It is known that every symmetric fusion category is braided equivalent to a category of the form $\mathrm{Rep}(G,\nu)$ where $G$ is a finite group and $\nu$ is a central element of $G$ of order at most $2$ (see \cite[Theorem 9.9.26]{tcat} and references within).  When $\nu=e$ is trivial, we say $\mathrm{Rep}(G)=\mathrm{Rep}(G,e)$ is \emph{Tannakian}, and when $\nu$ is nontrivial we say $\mathrm{Rep}(G,\nu)$ is \emph{super Tannakian}.  One should be very careful when dealing with equivalence classes of symmetrically braided fusion categories.  For example, the fusion category underlying the symmetrically braided fusion category $\mathrm{Rep}(S_3)$ has three distinct braidings: one is symmetrically braided while the other two are not (see Example \ref{braidfin} below).
\end{example}

\par Braided fusion categories equipped with spherical structures are called \emph{premodular fusion categories} and when the braiding is nondegenerate, \emph{modular fusion categories} \cite[Sections 8.13--8.14]{tcat}.  Premodular fusion categories benefit from numerical results such as the balancing equation \cite[Proposition 8.13.8]{tcat} while modular fusion categories benefit from the related representation theory of the modular group $\mathrm{SL}(2,\mathbb{Z})$ \cite[Section 8.16]{tcat}.  In particular, there are two $\mathrm{rank}(\mathcal{C})\times\mathrm{rank}(\mathcal{C})$ matrices $S=[S_{X,Y}]_{X,Y}$ and $T=\mathrm{diagonal}(\theta_X:X\in\mathcal{O}(\mathcal{C}))$ (also called \emph{twists} of simple objects) attached to a modular fusion category known as \emph{modular data} which has been used to productively to classify modular fusion categories of small rank (see \cite{ng2023classification} and references within).  The modular data of a modular fusion category determines the underlying fusion ring via the Verlinde formula \cite[Corollary 8.14.4]{tcat}, but does not determine a modular fusion category up to equivalence \cite{MR4254068}.

\begin{example}\label{exquant}
Many examples of premodular and modular fusion categories come from the representation theory of quantum groups at roots of unity which can be studied with an elementary understanding of Lie theory \cite{MR4079742}.  To each complex finite-dimensional simple Lie algebra $X$, indicated by its Dynkin label, and root of unit $q$ such that $q^2$ is a primitive $\ell$th root of unity with $\ell\in\mathbb{Z}$ which is greater than or equal to the dual Coxeter number of $\mathfrak{g}$, there exists a premodular fusion category $\mathcal{C}(X,\ell,q)$.  Many of these premodular categories and their subcategories are characterized by their fusion rules.  For example for any $N\in\mathbb{Z}_{\geq1}$, any spherical fusion category with the fusion rules of $\mathcal{C}(A_N,\ell,q)$ is tensor equivalent to $\mathcal{C}(A_N,\ell,q)$ for some $\ell$ and $q$ up to a twisting of the associators \cite{MR1237835}.  Similarly \cite[Theorem A.3]{MR4486913}, any spherical fusion category with the fusion rules of $\mathcal{C}(A_1,\ell,q)_\mathrm{ad}$ is tensor equivalent to $\mathcal{C}(A_1,\ell,q)_\mathrm{ad}$ for some choice of $\ell$ and $q$.
\end{example}

If $\mathcal{C}$ is a (spherical) fusion category then one can construct a (premodular) braided fusion category $\mathcal{Z}(\mathcal{C})$, the \emph{double} of $\mathcal{C}$, whose simple objects are sums of simple objects of $\mathcal{C}$ equipped with braiding isomorphisms \cite[Definition 7.13.1]{tcat} with all other objects of $\mathcal{C}$.  It is well-known that $\mathcal{Z}(\mathcal{C})$ is furthermore nondegenerately braided, i.e.\ modular in the presence of a spherical structure.  A principal tool that we will use is induction-restriction to the double \cite[Section 9.2]{tcat}.  The obvious forgetful functor $F:\mathcal{Z}(\mathcal{C})\to\mathcal{C}$ has an adjoint functor we call induction $I:\mathcal{C}\to\mathcal{Z}(\mathcal{C})$.  The structure of this functor is roughly controlled by the adjoint condition, and various other results related to the modular data of the double.  For example, it is well-known \cite[Proposition 9.2.2]{tcat} that for all $X\in\mathcal{O}(\mathcal{C})$,
\begin{equation}
F(I(X))=\bigoplus_{Y\in\mathcal{O}(\mathcal{C})}Y\otimes X\otimes Y^\ast.
\end{equation}
Some simple objects of $\mathcal{Z}(\mathcal{C})$ are determined by the \emph{formal codegrees} of the underlying fusion rules \cite{codegrees}.  Denote the set of finite-dimensional complex irreducible representations of the Grothendieck ring of $\mathcal{C}$ by $\mathrm{Irr}(\mathcal{C})$.  If $f_\varphi$ is a formal codegree of a spherical fusion category $\mathcal{C}$ corresponding to some $\varphi\in\mathrm{Irr}(\mathcal{C})$, there exists a simple object $A_\varphi\in\mathcal{O}(\mathcal{Z}(\mathcal{C}))$ such that $\dim(A_f)=\dim(\mathcal{C})/f_\varphi$ and $[\mathbbm{1},A_\varphi]=\dim(\varphi)$.  These simple objects constitute all simple summands of $I(\mathbbm{1})$ and all have trivial twists \cite[Proposition 3.11]{MR4655273}.  

\par Braidings on an arbitrary fusion category $\mathcal{C}$ can be described by the double construction.  Specifically, a braiding on $\mathcal{C}$ corresponds to a fusion subcategory $\mathcal{D}\subset\mathcal{Z}(\mathcal{C})$ with $\mathrm{FPdim}(\mathcal{D})=\mathrm{FPdim}(\mathcal{C})$ and $\mathcal{O}(\mathcal{D})\cap I(\mathbbm{1})=\{\mathbbm{1}\}$ \cite[Theorem 3.2]{MR3943750}.  We include the following example for the reader to illustrate that this correspondence may or may not be one-to-one.

\begin{example}\label{braidfin}
Consider all possible braidings on the rank $3$ spherical fusion category $\mathrm{Rep}(S_3)$, of finite-dimensional complex representations of the symmetric group $S_3$.  The formal codegrees \cite{codegrees} of $\mathrm{Rep}(S_3)$ are $6$, $3$, and $2$, and so $I(\mathbbm{1})\subset\mathcal{Z}(\mathrm{Rep}(S_3))$ has simple summands of dimensions $1$ (the tensor unit), $2$, and $3$.  The $S$-matrix of the double is
\begin{equation}
\left[\begin{array}{cccccccc}
1 & 1 & 2 & 2 & 2 & 2 & 3 & 3 \\
1 & 1 & 2 & 2 & 2 & 2 & -3 & -3 \\
2 & 2 & -2 & 4 & -2 & -2 & 0 & 0 \\
2 & 2 & 4 & -2 & -2 & -2 & 0 & 0 \\
2 & 2 & -2 & -2 & 4 & -2 & 0 & 0 \\
2 & 2 & -2 & -2 & -2 & 4 & 0 & 0 \\
3 & -3 & 0 & 0 & 0 & 0 & 3 & -3 \\
3 & -3 & 0 & 0 & 0 & 0 & -3 & 3
\end{array}\right]
\end{equation}
and $T=\mathrm{diagonal}(1,1,\omega^2,\omega,1,1,1,-1)$ where $\omega$ is a primitive third root of unity.  One may verify with the Verlinde formula \cite[Corollary 8.14.4]{tcat} that all simple objects of dimension $2$ $\otimes$-generate fusion subcategories $\mathcal{D}\subset\mathcal{Z}(\mathrm{Rep}(S_3))$ with $\mathrm{FPdim}(\mathcal{D})=6$.  Therefore three of these simple objects, of twists $1,\omega,\omega^2$ correspond to three distinct braidings on $\mathrm{Rep}(S_3)$, distinguished by twists, while the fourth simple object of dimension $2$ and trivial twist appears as a summand of $I(\mathbbm{1})$.  We denote these braided fusion categories $\mathrm{Rep}(S_3)^\omega$ for third roots of unity $\omega$ with $\mathrm{Rep}(S_3)^1=\mathrm{Rep}(S_3)$.  In this case the inequivalent braidings on the fusion category $\mathrm{Rep}(S_3)$ are truly in bijection with the fusion subcategories from \cite[Theorem 3.2]{MR3943750}.

\par Alternately, consider all possible braidings on the rank $4$ fusion category $\mathrm{Rep}(A_4)$.  The formal codegrees of $\mathrm{Rep}(A_4)$ are $12$, $4$, $3$ and $3$, and so $I(\mathbbm{1})\subset\mathcal{Z}(\mathrm{Rep}(A_4))$ has simple summands of dimensions $1$ (the tensor unit), $3$, $4$ and $4$.  With $\alpha:=-4\omega^2$, the $S$-matrix of the double is
\begin{equation}
\left[\begin{array}{cccccccccccccc}
1 & 1 & 1 & 3 & 4 & 4 & 4 & 3 & 3 & 3 & 3 & 4 & 4 & 4 \\
1 & 1 & 1 & 3 & 4\omega & 4\omega & 4\omega & 3 & 3 & 3 & 3 & \alpha & \alpha & \alpha \\
1 & 1 & 1 & 3 & \alpha & \alpha & \alpha & 3 & 3 & 3 & 3 & 4\omega & 4\omega & 4\omega \\
3 & 3 & 3 & 9 & 0 & 0 & 0 & -3 & -3 & -3 & -3 & 0 & 0 & 0 \\
4 & 4\omega & \alpha & 0 & 4 & 4\omega & \alpha & 0 & 0 & 0 & 0 & 4 & 4\omega & 4 \\
4 & 4\omega & \alpha & 0 & 4\omega & \alpha & 4 & 0 & 0 & 0 & 0 & 4 & 4 & 4\omega \\
4 & 4\omega & \alpha & 0 & \alpha & 4 & 4\omega & 0 & 0 & 0 & 0 & 4\omega & 4 & 4 \\
3 & 3 & 3 & -3 & 0 & 0 & 0 & 9 & -3 & -3 & -3 & 0 & 0 & 0 \\
3 & 3 & 3 & -3 & 0 & 0 & 0 & -3 & -3 & -3 & 9 & 0 & 0 & 0 \\
3 & 3 & 3 & -3 & 0 & 0 & 0 & -3 & -3 & 9 & -3 & 0 & 0 & 0 \\
3 & 3 & 3 & -3 & 0 & 0 & 0 & -3 & 9 & -3 & -3 & 0 & 0 & 0 \\
4 & \alpha & 4\omega & 3 & 4 & \alpha & 4\omega & 0 & 0 & 0 & 0 & 4 & \alpha & 4\omega \\
4 & \alpha & 4\omega & 3 & 4\omega & 4 & \alpha & 0 & 0 & 0 & 0 & \alpha & 4\omega & 4 \\
4 & \alpha & 4\omega & 3 & \alpha & 4\omega & 4 & 0 & 0 & 0 & 0 & 4\omega & 4 & \alpha
\end{array}\right]
\end{equation}
and $T=\mathrm{diagonal}(1,1,1,1,1,\omega^2,\omega,1,-1,1,-1,1,\omega,\omega^2)$.  One may verify with the Verlinde formula that all simple objects of dimension $3$ $\otimes$-generate fusion subcategories $\mathcal{D}\subset\mathcal{Z}(\mathrm{Rep}(A_4))$ with $\mathrm{FPdim}(\mathcal{D})=12$.  Four of these subcategories intersect $I(\mathbbm{1})$ trivially; one is Tannakian while the other three are braided equivalent.  We denote these two braided fusion categories $\mathrm{Rep}(A_4)^\epsilon$ for second roots of unity $\epsilon$ with $\mathrm{Rep}(A_4)^1=\mathrm{Rep}(A_4)$.  In this case the number of inequivalent braidings on the fusion category $\mathrm{Rep}(A_4)$ are fewer in number than the fusion subcategories from \cite[Theorem 3.2]{MR3943750}.
\end{example}


\subsection{Equivariantization and de-equivariantization}\label{sec:deeq}

Given a fusion category $\mathcal{C}$ and a finite group $G$, there are two related constructions to produce new fusion categories combining these objects.  One was touched on in Section \ref{ssec:ring} which is to create a $G$-graded extension $\mathcal{D}$ such that $\mathcal{D}_\mathrm{ad}=\mathcal{C}$.  The categorical obstructions to such an extension existing are outlined in \cite{ENO}.  Alternatively, if $G$ acts on $\mathcal{C}$ by tensor autoequivalences \cite[Definition 2.7.1]{tcat}, one can construct the fusion category $\mathcal{C}^G$ of $G$-equivariant objects in $\mathcal{C}$, or the \emph{equivariantization of $\mathcal{C}$ by $G$}.  The simple objects of $\mathcal{C}^G$ are roughly orbits of the $G$-action on $\mathcal{O}(\mathcal{C})$ along with representation-theoretic data, depending on the action of $G$ by tensor autoequivalences, related to the fixed-points of this action \cite[Remark 4.15.8]{tcat}.  In any case, $\mathrm{FPdim}(\mathcal{C}^G)=|G|\mathrm{FPdim}(\mathcal{C})$.  Conversely, if $\mathcal{C}$ is a fusion category and $\mathrm{Rep}(G)\to\mathcal{Z}(\mathcal{C})$ is a braided tensor functor such that its composition with the forgetful functor $F:\mathcal{Z}(\mathcal{C})\to\mathcal{C}$ is fully faithful, then there exists a fusion category $\mathcal{C}_G$, the \emph{de-equivariantization of $\mathcal{C}$ by $G$} and an action of $G$ on $\mathcal{C}_G$ such that $(\mathcal{C}^G)_G\simeq(\mathcal{C}_G)^G\simeq\mathcal{C}$ \cite[Theorem 8.23.3]{tcat}.  If $G$ in addition acts by braided autoequivalences, one can upgrade the equivariantization and de-equivariantization constructions to keep track of this additional structure as well.  We refer the reader to \cite[Section 4.2]{DGNO} for further information and details.

\par In what follows, we will mainly use equivariantization and de-equivariantization in the braided setting where $\mathcal{E}\simeq\mathrm{Rep}(G)\subseteq\mathcal{C}$ is a Tannakian fusion subcategory (see Section \ref{subsec:braid}).  One can then identify the de-equivariantization $\mathcal{C}_G$ with $\mathcal{C}_A$, the category of $A$-module objects in $\mathcal{C}$ where $A$ is the regular algebra of $\mathrm{Rep}(G)$.  In this case, $\mathcal{C}_G$ contains a braided fusion subcategory $\mathcal{C}_A^0\simeq(C_\mathcal{C}(\mathcal{E}))_A$.  Thus when $\mathcal{E}\subset C_\mathcal{C}(\mathcal{C})$ lies in the symmetric center of $\mathcal{C}$, the entire category of $A$-modules $\mathcal{C}_A=\mathcal{C}_A^0$ is braided.  Furthermore, $\mathcal{C}_A^0$ is nondegenerately braided if and only if $\mathcal{C}$ is.  The following example demonstrates how quickly the number of braided equivalence classes of braided fusion categories grows by using the equivariantization construction.

\begin{example}\label{ex:12} There are six equivalence classes of pointed fusion categories $\mathcal{C}$ of rank $6$, corresponding to products of the two inequivalent pointed fusion categories of rank $2$, and the three inequivalent pointed fusion categories of rank $3$ (see Example \ref{ex:p}).  Any tensor autoequivalence which acts nontrivially on the isomorphism classes of invertible objects must permute the elements of orders $3$ and $6$, while the invertible objects of orders $1$ and $2$ are fixed, so there is a unique $\tau\in\mathrm{Aut}_\otimes(\mathcal{C})$ of order $2$ up to isomorphism.  The only additional data of a nontrivial $C_2$-action on $\mathcal{C}$ by tensor autoequivalences is an isomorphism of tensor autoequivalences $\tau^2\cong\mathrm{id}_\mathcal{C}$, of which there are two, yielding twelve equivalence classes of fusion categories $\mathcal{C}^{C_2}$ whose fusion rules are those of the character rings of two of the groups of order $12$: $D_6\cong C_2\times S_3$ or $\mathrm{Dic}_3$ \cite{MR3059899}.

\par But only four of these fusion categories are braided (see Example \ref{ex:braidgroup}).  Indeed, a quadratic form on $C_2\times C_3$ is determined by its values on each factor; there are four quadratic forms on $C_2$ whose unique nontrivial values are the fourth roots of unity, while a quadratic form on $C_3$ is determined by the choice of a single third root of unity, the other nontrivial value is the same by duality.  We will denote the associated premodular fusion categories by $\mathcal{C}(C_6,q_{x,\omega})$ where $x$ is any fourth root of unity and $\omega$ is any third root of unity.  These twelve quadratic forms correspond to only two fusion categories up to tensor equivalence, distinguished by whether the quadratic form on the $C_2$ factor is nondegenerate (see Example \ref{ex:braidgroup}).  As noted above, there are two distinct $C_2$-actions on $\mathcal{C}(C_6,q_{x,\omega})$ by tensor autoequivalences, both of which are braided, giving four inequivalent fusion categories by this construction possessing braidings, but $24$ braided fusion categories up to braided equivalence by starting with various quadratic forms on $C_2\times C_3$.  The twists defined by the canonical spherical structure of these premodular fusion categories distinguish these categories, displayed in Figure \ref{fig:dim12}.  We leave it as an exercise to the reader to repeat the above arguments for the trivial $C_2$-action on $\mathcal{C}(C_6,q_{x,\omega})$ which produces only a portion of the the pointed braided fusion categories of Frobenius-Perron dimension $12$; the remainder can be identified directly from quadratic forms on abelian groups of order $12$.

\begin{figure}[H]
\centering
\begin{align*}
\begin{array}{|c|c|c|c|}
\hline \mathcal{C}(C_6,q_{x,\omega})^{C_2} & \theta & \# \\\hline\hline
\mathrm{Rep}(C_2,\nu)\boxtimes\mathrm{Rep}(S_3)^\omega &  1,1,\nu,\nu,\omega,\nu\omega  & 6 \\
\mathrm{Rep}(\mathrm{Dic}_3,\nu)^\omega & 1,1,\nu,\nu,\nu\omega,\nu\omega & 6 \\\hline\hline
\mathcal{C}(C_2,\pm i)\boxtimes\mathrm{Rep}(S_3)^\omega  & 1,1,\pm i,\pm i,\omega,\pm i\omega & 6 \\
\mathcal{C}(C_6,q_{\pm i,\omega})^{C_2} & 1,1,\pm i,\pm i,\pm i\omega,\pm i\omega & 6\\\hline
\end{array}
\end{align*}
    \caption{The $24$ braided categorifications, up to braided equivalence, of the character rings of $G:=C_2\times S_3$ and $H:=\mathrm{Dic}_3$, separated by the underlying fusion category being $\mathrm{Rep}(G)$ or $\mathrm{Rep}(H)$ (top), or not (bottom).}%
    \label{fig:dim12}%
\end{figure}
\end{example}

\begin{example}\label{ex:12222}
Note that there are exactly two inequivalent pointed fusion categories of rank $9$ which admit braidings, i.e.\ $\mathrm{Vec}_{C_3^2}$ and $\mathrm{Vec}_{C_9}$, with trivial associators, and a unique fixed-point free $C_2$-action on each by tensor autoequivalences via $X\mapsto X^\ast$ for all simple $X$.  This produces the fusion categories $(\mathrm{Vec}_{C_3^2})^{C_2}\simeq\mathrm{Rep}(C_3\rtimes S_3)$ and $(\mathrm{Vec}_{C_9})^{C_2}\simeq\mathrm{Rep}(D_9)$, respectively, via $C_2$-equivariantization.  Each of $\mathrm{Vec}_{C_3^2}$ and $\mathrm{Vec}_{C_9}$ admits five inequivalent braidings.  In other words, there are $10$ inequivalent pre-metric groups over all $\mathcal{C}(C_3^2,q)$ and $\mathcal{C}(C_9,q)$.  The former are all distinct products over the pre-metric groups $\mathcal{C}(C_3,q)$ and the latter $q:C_9\to\mathbb{C}^\times$ either take values in third roots of unity ($3$ braided inequivalent classes) or ninth roots of unity ($2$ braided inequivalent classes).
\end{example}


\section{Near-integral fusion}\label{sec:nearfusion}

\subsection{Definition and examples}\label{subsec:def}

We say that a fusion ring $(R,B)$ is \emph{near-integral} if there exists a proper fusion subring of maximal rank, i.e.\ a fusion ring $(S,C)$ with $C\subset B$ such that $|C|+1=|B|$.  Let $\rho$ be the unique element in $B\setminus C$.  It is clear that $\rho=\rho^\ast$ since $C$ is closed under duality.  As $S$ is a fusion subring, for any $x,y\in C$,
\begin{equation}
0=c_{x^\ast,y}^\rho=c_{y,\rho}^x
\end{equation}
by \cite[Proposition 3.1.6]{tcat}. Thus for any $x\in C$, $c_{x,\rho}^\rho=\mathrm{FPdim}(x)$.  Therefore $S$ is an integral fusion ring, and the only fusion rule that is undetermined is $\kappa:=c_{\rho,\rho}^\rho\in\mathbb{Z}_{\geq0}$.  Moreover, $R$ is commutative if and only if $S$ is commutative.  We denote such a fusion ring $R(S,\kappa)$ and the distinguished nontrivial basis element will be denoted by $\rho$.  We fix the notation $d_+:=\mathrm{FPdim}(\rho)$ and $N:=\mathrm{FPdim}(S)$.  It is clear that $d_+$ satisfies
\begin{equation}\label{eq:def}
x^2-\kappa x-N=0.
\end{equation}
Note also that the product of the solutions of Equation (\ref{eq:def}) is $-N$ so the solution which is not $d_+=\mathrm{FPdim}(\rho)$ is negative; we denote this root $d_-$.

\begin{example}
The near-integral fusion rings $R(\mathbb{Z}G,\kappa)$ for a finite group $G$ and their categorifications have been studied under the name \emph{near-group} fusion rings, initiated in \cite{MR1997336}, and were inspired by the rings $R(\mathbb{Z}G,0)$ which are known as the \emph{Tambara-Yamagami} fusion rings after seminal work in \cite{MR1659954}.
\end{example}

\begin{lemma}\label{bund}
Let $R(S,\kappa)$ be a near-integral fusion ring.  If $d_+\in\mathbb{Z}$, then $\kappa<N$.  Moreover if $\kappa\geq N$, then $d_+$ is irrational.
\end{lemma}

\begin{proof}
Equation (\ref{eq:def}) states that $N=d_+(d_+-\kappa)$.  If $d_+$ is an integer, then $d_+>\kappa$ since $N\in\mathbb{Z}_{\geq1}$.  Moreover $\kappa<d_+\leq N$.
\end{proof}

\subsection{Representation theory}\label{subsec:rep}

\begin{proposition}\label{prop:char}
Let $R(S,\kappa)$ be a near-integral fusion ring.  There exist exactly two ring homomorphisms $\chi_\pm:R(S,\kappa)\to\mathbb{C}$ such that $\chi_\pm(\rho)\neq0$.  In particular, $\left.\chi_\pm\right|_S=\mathrm{FPdim}$ and $\chi_\pm(\rho)=d_\pm$.
\end{proposition}

\begin{proof}
Let $\chi:R\to\mathbb{C}$ be any ring homomorphism such that $\chi(\rho)\neq0$.  For any $x\in C$, $\rho x=\mathrm{FPdim}(x)\rho$.  Therefore $\chi(\rho)\neq0$ implies $\chi(x)=\mathrm{FPdim}(x)$ for all $x\in C$.  Then we compute
\begin{equation}
\chi(\rho)^2=\chi(\rho^2)=\kappa\chi(\rho)+\sum_{x\in C}\mathrm{FPdim}(x)\chi(x)=\kappa\chi(\rho)+N.
\end{equation}
Therefore $\chi(\rho)$ is one of the two solutions of Equation (\ref{eq:def}). 
\end{proof}

\begin{proposition}\label{prop:char2}
Let $R(S,\kappa)$ be a near-integral fusion ring.  Any $\varphi\in\mathrm{Irr}(S)$ with $\varphi\neq\mathrm{FPdim}$ extends uniquely to $\tilde{\varphi}\in\mathrm{Irr}(R)$ such that $\tilde{\varphi}(\rho)=0$. 
\end{proposition}

\begin{proof}
Let $\varphi:S\to\mathrm{End}(V)$ be an irreducible representation of $S$ and define $\tilde{\varphi}:R\to\mathrm{End}(V)$ such that $\tilde{\varphi}|_S=\varphi$ and $\tilde{\varphi}(\rho)=0$.  The function $\tilde{\varphi}$ is evidently additive and the only nontrivial verification to ensure $\tilde{\varphi}$ is multiplicative is that $\tilde{\varphi}(\rho^2)=0$.  To this end, we compute
\begin{equation}\label{tooo}
\tilde{\varphi}(\rho^2)=\kappa\tilde{\varphi}(\rho)+\sum_{x\in C}\mathrm{FPdim}(x)\tilde{\varphi}(x)=\sum_{x\in C}\mathrm{FPdim}(x)\varphi(x^\ast)=0
\end{equation}
where the final equality follows from \cite[Lemma 2.3]{ost15}.  Thus $\tilde{\varphi}\in\mathrm{Rep}(R)$.  It is evident that $\tilde{\varphi}\in\mathrm{Irr}(R)$ as any nontrivial decomposition of $\tilde{\varphi}\in\mathrm{Irr}(R)$ would correspond to a nontrivial decomposition of $\varphi\in\mathrm{Irr}(S)$ into irreducible representations.  In sum,
\begin{equation}
\dim(\chi_+)^2+\dim(\chi_-)^2+\sum_{\substack{\varphi\in\mathrm{Irr}(S) \\ \varphi\neq\mathrm{FPdim}}}\dim(\varphi)^2=1+1+\mathrm{rank}(S)-1=\mathrm{rank}(R),
\end{equation}
and so this collection completely describes $\mathrm{Irr}(R)$.
\end{proof}

\begin{corollary}\label{cor:form}
Let $R(S,\kappa)$ be a near-integral fusion ring.  The formal codegrees of $R(S,\kappa)$ are the set of formal codegrees $f$ of $S$ such that $f\neq\mathrm{FPdim}(S)$ along with $\mathrm{FPdim}(S)+d_\pm^2$.
\end{corollary}

\subsection{Characterization}\label{subsec:char}

\begin{theorem}\label{thm:char}
Let $(R,B)$ be a fusion ring.  Then $R$ is near-integral if and only if there exists a ring homomorphism $\chi:R\to\mathbb{C}$ and $\rho\in B$ such that $\chi(\rho)\neq\mathrm{FPdim}(\rho)$ and $\chi(x)=\mathrm{FPdim}(x)$ for all $x\neq\rho$.
\end{theorem}

\begin{proof}
The forward direction follows from Proposition \ref{prop:char}.  The converse direction follows by considering the \emph{kernel} \cite[Definition 3.1]{MR3552799}
\begin{equation}
C:=B\setminus\{\rho\}=\{x\in B:\chi(x)=\mathrm{FPdim}(x)\}.
\end{equation}
Note that if $x,y\in C$, then by the triangle inequality,
\begin{align}
\mathrm{FPdim}(xy)=\chi(xy)&=\left|c_{x,y}^\rho\chi(\rho)+\sum_{z\in C}c_{x,y}^z\mathrm{FPdim}(z)\right| \\
&\leq|c_{x,y}^\rho\chi(\rho)|+\sum_{z\in C}c_{x,y}^z\mathrm{FPdim}(z) \\
&\leq\mathrm{FPdim}(xy)
\end{align}
since $\chi(\rho)$ is an eigenvalue of the matrix of (left) multiplication by $\rho$.  The triangle inequality implies $\chi(\rho)=\mathrm{FPdim}(\rho)$ if $c_{x,y}^\rho\neq0$.  But $\chi\neq\mathrm{FPdim}$ by assumption, thus $c_{x,y}^\rho=0$.  Moreover $C$ is closed under fusion, hence $S\subset R$, the $\mathbb{Z}$-linear span of the basis $C$, is a fusion subring.  This implies $R\cong R(S,c_{\rho,\rho}^\rho)$ is near-integral as we aimed to prove.
\end{proof}

\begin{example}
Let $G$ be a finite group.  The character ring $R_G$ is commutative, and the (1-dimensional) irreducible representations of $R_G$ are given by columns of the character table of $G$.  Theorem \ref{thm:char} implies $R_G$ has near-integral fusion if and only if there exists a conjugacy class $x$ of $G$ and unique irreducible representation $\rho$ of $G$ whose character $\chi_\rho$ satisfies $\chi_\rho(x)\neq\dim(\rho)$ and $\chi_\varphi(x)=\dim(x)$ for all $\varphi\neq\rho$.  Column orthogonality with the trivial class states that
\begin{equation}
0=\sum_{\varphi\in\mathrm{Irr}(G)}\dim(\varphi)\chi_\varphi(x)=\sum_{\varphi\neq\rho}\dim(\varphi)^2+\dim(\rho)\chi_\rho(x).
\end{equation}
Hence
\begin{align}
\chi_\rho(x)=\dfrac{-1}{\dim(\rho)}\sum_{\varphi\neq\rho}\dim(\varphi)^2&=\dfrac{-1}{\dim(\rho)}(|G|-\dim(\rho)^2) \\
&=\dim(\rho)-\dfrac{|G|}{\dim(\rho)}.
\end{align}
But $\chi_\rho(x)^2=\kappa\chi_\rho(x)+|G|-\dim(\rho)^2$, hence
\begin{equation}\label{kappaeq}
\kappa=2\dim(\rho)-\dfrac{|G|}{\dim(\rho)}.
\end{equation}
The smallest example is $C_2$, the cyclic group of order $2$, whose character ring is a Tambara-Yamagami fusion ring; indeed $\kappa=2\cdot1-2/1=0$.  A non-trivial example of such $G$ is the group $M_9\cong\mathrm{PSU}(3,2)$.  This is two examples in one since the maximal fusion subring of $R_{M_9}$ is the Tambara-Yamagami fusion ring $R_{Q_8}\cong R(\mathbb{Z}C_2^2,0)$.  Thus $R_{M_9}\cong R(R(\mathbb{Z}C_2^2,0),7)$.  We include the character table for this example in Figure \ref{fig:repf5}.  Moreso, this example has a related near-group example, $R_{F_9}\cong R(\mathbb{Z}C_8,7)$ which lives in the infinite family of near-group examples coming from the Frobenius groups $F_q$ with $q$ a prime power.  We include a variety of small-rank illustrative examples in Figures \ref{fig:repc2} (non self-dual, near-group, etc.) in Figure \ref{fig:rep2}--\ref{fig:repc23c4}, some of which will be referenced in Section \ref{sec:brad}.

\begin{figure}[H]
\centering
\begin{equation*}
\begin{array}{|cc|}
\hline1 & 1 \\
1 & -1  \\
\hline
\end{array}
\qquad
\begin{array}{|ccc|}
\hline1 & 1 & 1 \\
1 & 1 & -1  \\
2 & -1 & 0\\
\hline
\end{array}
\qquad
\begin{array}{|cccc|}
\hline1 & 1 & 1 & 1 \\
1 & 1 & \zeta_3 & \zeta_3^2 \\
1 & 1 & \zeta_3^2 & \zeta_3 \\
3 & -1 & 0 & 0 \\
\hline
\end{array}
\qquad
\begin{array}{|ccccc|}
\hline1 & 1 & 1 & 1 & 1\\
1 & 1 & -1 & 1 & -1 \\
1 & 1 & 1 & -1  & -1 \\
1 & 1 & -1 & -1 & 1 \\
2 & -2 & 0 & 0 & 0\\
\hline
\end{array}
\end{equation*}
    \caption{Character tables of $C_2$, $S_3$, $A_4$, and $D_4$ \& $Q_8$}%
    \label{fig:repc2}%
\end{figure}

\begin{figure}[H]
\centering
\begin{equation*}
\begin{array}{|ccccc|}
\hline1 & 1 & 1 & 1 & 1 \\
1 & 1 & -1 & -1 & 1 \\
1 & 1 & -\zeta_4 & \zeta_4 & -1 \\
1 & 1 & \zeta_4 & -\zeta_4 & -1 \\
4 & -1 & 0 & 0 & 0 \\
\hline
\end{array}
\qquad
\begin{array}{|cccccc|}
\hline1 & 1 & 1 & 1 & 1 & 1 \\
1 & 1 & 1 & -1 & 1 & -1 \\
1 & 1 & 1 & 1 & -1 & -1 \\
1 & 1 & 1 & -1 & -1 & 1 \\
2 & 2 & -2 & 0 & 0 & 0 \\
8 & -1 & 0 & 0 & 0 & 0\\\hline
\end{array}
\end{equation*}
    \caption{Character tables of $F_5$ and $\mathrm{PSU}(3,2)$}%
    \label{fig:repf5}%
\end{figure}

\begin{figure}[H]
\centering
\begin{equation*}
\begin{array}{|cccccccccc|}
\hline1 & 1 & 1 & 1 & 1 & 1 & 1 & 1 & 1 & 1 \\
1 & 1 & -1 & 1 & 1 & -1 & -1 & 1 & 1 & 1 \\
1 & 1 & -1 & \zeta_3 & \zeta_3^2 & \zeta_6 & \zeta_6^5 & \zeta_3^2 & 1 & \zeta_3 \\
1 & 1 & 1 & \zeta_3^2 & \zeta_3 & \zeta_3 & \zeta_3^2 & \zeta_3 & 1 & \zeta_3^2  \\
1 & 1 & 1 & \zeta_3 & \zeta_3^2 & \zeta_3^2 & \zeta_3 & \zeta_3^2 & 1 & \zeta_3  \\
1 & 1 & -1 & \zeta_3^2 & \zeta_3 & \zeta_6^5 & \zeta_6 & \zeta_3 & 1 & \zeta_3^2 \\
2 & 2 & 0 & 2 & 2 & 0 & 0 & -1 & -1 & -1 \\
2 & 2 & 0 & \alpha & \bar{\alpha} & 0 & 0 & \zeta_6 & -1 & \zeta_6^5 \\
2 & 2 & 0 & \bar{\alpha} & \alpha & 0 & 0 & \zeta_6^5 & -1 & \zeta_6   \\
6 & -3 & 0 & 0 & 0 & 0 & 0 & 0 & 0 & 0 \\\hline
\end{array}
\end{equation*}
    \caption{Character table of $\mathrm{Aut}(D_9)$ with $\alpha:=-1+\sqrt{-3}$}%
    \label{fig:rep2}%
\end{figure}

\begin{figure}[H]
\centering
\begin{equation*}
\begin{array}{|ccccccccccc|}
\hline
1 & 1 & 1 & 1 & 1 & 1 & 1 & 1 & 1 & 1 & 1  \\
1 & 1 & 1 & 1 & 1 & -1 & 1 & -1 & -1 & 1 & -1  \\
1 & 1 & 1 & 1 & 1 & 1 & -1 & 1 & -1 & -1 & -1  \\
1 & 1 & 1 & 1 & 1 & -1 & -1 & -1 & 1 & -1 & 1  \\
1 & 1 & -1 & -1 & 1 & 1 & \zeta_4 & -1 & -\zeta_4 & -\zeta_4 & \zeta_4   \\
1 & 1 & -1 & -1 & 1 & -1 & \zeta_4 & 1 & \zeta_4 & -\zeta_4 & -\zeta_4   \\
1 & 1 & -1 & -1 & 1 & 1 & -\zeta_4 & -1 & \zeta_4 & \zeta_4 & -\zeta_4  \\
1 & 1 & -1 & -1 & 1 & -1 & -\zeta_4 & 1 & -\zeta_4 & \zeta_4 & \zeta_4   \\
2 & 2 & -2 & 2 & -2 & 0 & 0 & 0 & 0 & 0 & 0  \\
2 & 2 & 2 & -2 & -2 & 0 & 0 & 0 & 0 & 0 & 0  \\
4 & -4 & 0 & 0 & 0 & 0 & 0 & 0  & 0 & 0 & 0  \\
\hline
\end{array}
\end{equation*}
    \caption{Character table of $C_2^3\rtimes C_4$, $\mathrm{Smallgroup}(32,7)$, $\mathrm{Smallgroup}(32,8)$}%
    \label{fig:repc23c40}%
\end{figure}

\begin{figure}[H]
\centering
\begin{equation*}
\begin{array}{|ccccccccccc|}
\hline
1 & 1 & 1 & 1 & 1 & 1 & 1 & 1 & 1 & 1 & 1  \\
1 & 1 & 1 & 1 & 1 & -1 & 1 & -1 & 1 & -1 & -1  \\
1 & 1 & -1 & -1 & 1 & 1 & -1 & -1 & 1 & 1 & -1  \\
1 & 1 & -1 & -1 & 1 & -1 & -1 & 1 & 1 & -1 & 1  \\
1 & 1 & -1 & -1 & 1 & 1 & 1 & -1 & -1 & -1 & 1   \\
1 & 1 & -1 & -1 & 1 & -1 & 1 & 1 & -1 & 1 & -1   \\
1 & 1 & 1 & 1 & 1 & 1 & -1 & 1 & -1 & -1 & -1  \\
1 & 1 & 1 & 1 & 1 & -1 & -1 & -1 & -1 & 1 & 1   \\
2 & 2 & -2 & 2 & -2 & 0 & 0 & 0 & 0 & 0 & 0  \\
2 & 2 & 2 & -2 & -2 & 0 & 0 & 0 & 0 & 0 & 0  \\
4 & -4 & 0 & 0 & 0 & 0 & 0 & 0  & 0 & 0 & 0  \\
\hline
\end{array}
\end{equation*}
    \caption{Character table of $C_8\rtimes C_2^2$, $\mathrm{Smallgroup}(32,44)$}%
    \label{fig:repc23c4}%
\end{figure}

\end{example}


\section{Categorical near-integral fusion}\label{seccat}

\subsection{Basic results}

Let $\mathcal{C}$ be a fusion category such that its Grothendieck ring is the near-integral fusion ring $R(S,\kappa)$.  Denote $N:=\mathrm{FPdim}(S)$, $\rho\in\mathcal{O}(\mathcal{C})$ the simple object corresponding to the distinguished basis element $\rho\in R$, $\mathcal{D}\subset\mathcal{C}$ the fusion subcategory corresponding to $S$, and recall the roots $d_\pm$ from Equation \ref{eq:def}.  A generic fusion category with near-integral fusion rules will then be denoted $\mathcal{C}(\mathcal{D},\kappa)$ for an integral fusion category $\mathcal{D}$ and $\kappa\in\mathbb{Z}_{\geq0}$.  The abbreviated notation
\begin{equation}
[Y,Z]:=\dim_\mathbb{C}\mathrm{Hom}_\mathcal{C}(Y,Z)
\end{equation}
for any $Y,Z\in\mathcal{C}$ will be used in this section where the ambient category $\mathcal{C}$ is implied by the objects $Y,Z$ themselves which will be indicated clearly. 

\par We first show that the notions of categorical and Frobenius-Perron dimensions are essentially the same for near-integral fusion categories, so that we may use the characters $\mathrm{dim}$ and $\mathrm{FPdim}$ interchangeably without loss of generality (see \cite[Section 9.5]{tcat}).

\begin{lemma}\label{cor:1}
Any near-integral fusion category is Galois conjugate to a pseudo-unitary near-integral fusion category.  In particular, every near-integral fusion category has a canonical spherical structure.
\end{lemma}

\begin{proof}
If a near-integral fusion category is integral, it is pseudounitary \cite[Proposition 9.6.5]{tcat}.  Otherwise, Proposition \ref{prop:char} implies that $\dim$, which is a character of the Grothendieck ring of $\mathcal{C}$ \cite[Proposition 4.7.12]{tcat} since $\mathcal{C}$ is spherical \cite[Theorem 4.5]{MR4413278}, and $\mathrm{FPdim}$ are Galois conjugate as both take non-zero values on all simple objects.  The fact that $\dim$ takes non-zero values was proven in \cite[Theorem 2.3]{ENO}, while $\mathrm{FPdim}$ taking non-zero values follows from fusion matrices not being nilpotent.
\end{proof}

\par Recall the notions related to the double construction of a (spherical) fusion category from Section \ref{sec:brad}. For near-integral fusion categories $\mathcal{C}$, the most important simple object in $\mathcal{Z}(\mathcal{C})$ is $A_{\chi_-}\subset I(\mathbbm{1})$ corresponding to the character $\chi_-$ defined in Proposition \ref{prop:char}.  We compute
\begin{align}\label{furfteen}
\dfrac{2N+\kappa d_+}{N}&=\dfrac{N+d_+^2}{-d_-d_+}=1-\dfrac{d_+}{d_-}=2+\dfrac{\kappa}{N}d_+,
\end{align}
thus
\begin{equation}\label{thurteen}
\dim(A_{\chi_-})=\dfrac{\dim(\mathcal{C})}{f_{\chi_-}}=\dfrac{2N+\kappa d_+}{2N+\kappa d_-}=-\dfrac{d_+}{d_-}=\dfrac{d^2}{N}=1+\dfrac{\kappa}{N}d_+.
\end{equation}

\begin{lemma}\label{lem:divides0}
Let $\mathcal{C}$ be a near-integral fusion category.  If $\mathrm{FPdim}(\mathcal{C})\in\mathbb{Z}$, then $[F(A_{\chi_-}),\rho]=0$.\end{lemma}

\begin{proof}
We must have $\kappa<N$ if $d_+\in\mathbb{Z}$ by Lemma \ref{bund}, thus $\dim(A_{\chi_-})=1+(\kappa/N)d_+<1+d_+$ and the result follows.  If $d_+\not\in\mathbb{Z}$ and $\mathrm{FPdim}(\mathcal{C})\in\mathbb{Z}$, then $\kappa=0$ hence $[F(A_{\chi_-}),\rho]\leq[F(\mathbbm{1}),\rho]=0$ as well.
\end{proof}

\begin{lemma}\label{lem:divides}
Let $\mathcal{C}$ be a near-integral fusion category with $\mathrm{FPdim}(\mathcal{C})$ irrational.  Then $N\mid\kappa$.
\end{lemma}

\begin{proof}
By Equation (\ref{thurteen}), there exists a simple object $A_{\chi_-}$ in $\mathcal{Z}(\mathcal{C})$ with $F(A_{\chi_-})=\mathbbm{1}\oplus Y$ for some $Y\in\mathcal{C}$ with $\dim(Y)=(\kappa/N)d_+$.  If $d_+\not\in\mathbb{Z}$, this implies $F(A_{\chi_-})=\mathbbm{1}\oplus(\kappa/N)\rho$, hence $N\mid\kappa$.
\end{proof}

\begin{proposition}\label{propcomm}
Let $\mathcal{C}$ be a near-integral fusion category.  If $\mathrm{FPdim}(\mathcal{C})\not\in\mathbb{Z}$, then the fusion rules of $\mathcal{C}$ are commutative.
\end{proposition}

\begin{proof}
Assume $\mathcal{C}:=\mathcal{C}(\mathcal{D},\kappa)$.  As $\mathrm{FPdim}(\mathcal{C})\not\in\mathbb{Z}$, there exists $k\in\mathbb{Z}_{\geq1}$ such that $\kappa=kN$ where $N:=\mathrm{FPdim}(\mathcal{D})$.  We may assume without loss of generality that $\mathcal{C}$ is pseudounitary, hence we may equip $\mathcal{C}$ with its canonical spherical structure such that $\mathrm{FPdim}=\dim$.  Let $\varphi\in\mathrm{Irr}(\mathcal{D})$ such that $\varphi\neq\mathrm{FPdim}$.  By Proposition \ref{prop:char2} we may consider $\varphi$ as an irreducible representation of $\mathcal{C}$ via extending $\varphi$ to $\rho$ by zero.  The formal codegree $f_\varphi$ is a positive integer by \cite[Corollary 2.15]{ost15} which divides $N$, hence there exist positive integers $g_\varphi:=N/f_\varphi$ and simple objects $A_\varphi\in\mathcal{O}(\mathcal{Z}(\mathcal{C}))$ in the induction $I(\mathbbm{1})$ with
\begin{equation}
\dim(A_\varphi)=\dfrac{\dim(\mathcal{C})}{f_\varphi}=g_\varphi\dfrac{\dim(\mathcal{C})}{N}=g_\varphi\left(2+kd_+\right)
\end{equation}
by Equation (\ref{furfteen}).  Hence $[F(A_\varphi),\rho]=kg_\varphi$  Recall \cite[Proposition 2.10]{ost15} implies 
\begin{equation}
\sum_{\varphi\in\mathrm{Irr}(\mathcal{D})}g_\varphi=\sum_{\varphi\in\mathrm{Irr}(\mathcal{D})}\dfrac{N}{f_\varphi}\leq N\sum_{\varphi\in\mathrm{Irr}(\mathcal{D})}\dfrac{\dim(\varphi)}{f_\varphi}=N,
\end{equation}
hence $\sum_{\varphi\neq\mathrm{FPdim}}g_\varphi\leq N-1$ with equality if and only if $\dim(\varphi)=1$ for all $\varphi\neq\mathrm{FPdim}$.   The simple objects $A_{\chi_\pm}$ are the only other simple summands of $I(\mathbbm{1})$ not of the form $A_\varphi$ for some $\varphi\in\mathrm{Irr}(\mathcal{D})$ with $\varphi\neq\mathrm{FPdim}$, hence we have
\begin{align}
kN=[F(I(\mathbbm{1})),\rho]&=[F(A_{\chi_+}\oplus A_{\chi_-}),\rho]+\sum_{\varphi\neq\mathrm{FPdim}}[I(\mathbbm{1}),A_\varphi][F(A_\varphi),\rho] \\
&=k\left(1+\sum_{\varphi\neq\mathrm{FPdim}}[I(\mathbbm{1}),A_\varphi]\dfrac{N}{f_\varphi}\right) \\
&\geq k(1+ (N-1)) \\
&=kN
\end{align}
with equality if and only if $\dim(\varphi)=[\mathbbm{1},F(A_\varphi)]=[I(\mathbbm{1}),A_\varphi]=1$ for all $\varphi\neq\mathrm{FPdim}$, thus $\mathcal{D}$ is commutative and moreover $\mathcal{C}$ is commutative as well.
\end{proof}


\subsection{Braided near-integral fusion categories}\label{sec:brad}

\par Braided near-group fusion categories, i.e.\ those for which $\mathcal{D}$ is pointed, have been explicitly described.  The main reason a complete classification is possible is that there is a very succinct description of symmetrically braided near-group fusion categories, which is to say finite groups $G$ having exactly one nonlinear irreducible character \cite{MR222160}.  Subsequently, braidings on $C_2$-graded near-group categories were classified in \cite{siehler} (see also \cite[Corollary 4.10]{MR4484228}).  Both of these contain infinite families of examples, while there are exactly $7$ braided equivalence classes of braided near-group fusion categories which are neither symmetrically braided nor nontrivially graded \cite[Theorem III.4.6]{thornton2012generalized}.

\par Braided near-integral fusion categories are more numerous.  Those which are symmetrically braided were first studied in \cite{MR721927} under the guise of finite groups with a character vanishing on all but $2$ conjugacy classes.  Proposition \ref{prop:char2} and Theorem \ref{thm:char} imply near-integral fusion rings are a true generalization of this definition.  Here we provide a diverse collection of examples.

\begin{example}\label{ex:exspecial}
Let $p\in\mathbb{Z}_{\geq2}$ and $G$ be a finite $p$-group.  The group $G$ is \emph{extraspecial} if $|Z(G)|=p$ and $G/Z(G)$ is a nontrivial elementary abelian $p$-group.  For each prime $p$ and positive integer $n$, there exist exactly two isomorphism classes of extraspecial $p$-groups of order $p^{2n+1}$, traditionally denoted $p_\pm^{1+2n}$.  For the duration, let $p$ be an odd prime, in which case $p_-^{1+2n}$ is distinguished as having elements of order $p^2$, while $p_+^{1+2n}$ has only elements of order $p$.

\par The example of $\mathrm{Aut}(D_9)$ in Figure \ref{fig:rep2} can be generalized by considering $\mathrm{Aut}(D_9)$ as the semidirect product $3_-^{1+2}\rtimes C_2$.  To this end, let $G$ be an extraspecial $p$-group for an odd prime $p$ and $z\in Z(G)$ be a generator.  The group $G$ can be realized by generators $x_1,\ldots,x_{2n}$ subject to the relations
\begin{align}
[x_{2i-1},x_{2i}]&=z&&\text{for }1\leq i\leq n \\
[x_i,x_j]&=1&&\text{for }1\leq i,j\leq n\text{ and }|i-j|>1,\text{ and } \\
x_i^p&\in Z(G)&&\text{for }1\leq i\leq2n.
\end{align}
Now let $k\in\mathbb{Z}$ with $1<k<p$.  We define a map $\theta:G\to G$ by $\theta(z)=z^k$, $\theta(x_{2i-1})=x_{2i-1}^k$, and $\theta(x_{2i})=x_{2i}$ for $1\leq i\leq n$.  It is straightforward to verify that $\theta\in\mathrm{Aut}(G)$ and $\theta$ has order $p-1$.  Moreover, $\mathrm{Aut}(G)\cong H\rtimes\langle\theta\rangle$ where $H$ is the subgroup of $\mathrm{Aut}(p_\pm^{1+2n})$ fixing $Z(p_\pm^{1+2n})$ pointwise.

\par Now recall that the isomorphism classes of irreducible representations of $p_\pm^{1+2n}$ consist of $2n$ one-dimensional representations which act trivially on $Z(G)$, and $p-1$ representations of dimension $p^n$ which are distinguished by their action on $Z(G)$.  In particular, $\langle\theta\rangle$ gives a well-defined $C_{p-1}$-action on $\mathrm{Rep}(G)$ by braided autoequivalences (refer to \cite[Section 2.3]{MR3390789}, for example).  The equivariantization $\mathrm{Rep}(G)^{C_{p-1}}$ has a unique simple object $\rho$ with $\dim(\rho)=p^n(p-1)$ while all other simple objects $X$ have $\dim(X)\leq p-1$, hence $\dim(X)^2\leq(p-1)^2<p(p-1)\leq\dim(\rho)$.  Moreover $\mathrm{Rep}(G)^{C_{p-1}}$ has a maximal fusion subcategory.  We compute
\begin{equation}
\kappa=2p^n(p-1)-\dfrac{p^{2n+1}(p-1)}{p^n(p-1)}=p^n(p-2),
\end{equation}
while the maximal fusion subcategory has dimension $N=p^{2n}(p-1)$.
\end{example}

\begin{note}
Note that one can see the above examples as a generalization of the \emph{near-group} braided fusion categories $\mathrm{Rep}(F_p)$ where $F_p$ is the Frobenius group of order $p(p-1)$.  In particular, one can view $C_p$ as the degenerate extra-special $p$-group where $n=0$.
\end{note}

\begin{note}
Note also that these examples are finite groups with the (relatively) largest possible irreducible character degrees.  In particular, if $G$ is a nontrivial finite group and $d$ is the degree of an irreducible character of $G$, then $|G|=d(d+e)$ for some $e\in\mathbb{Z}_{\geq1}$ \cite[Theorem 1.1]{MR3458181}.  It is known that $|G|\leq e^3(e-1)$, and the semidirect products above achieve this maximal bound.
\end{note}

To study nonsymmetrically braided near-integral fusion categories, there are two fundamentally distinct cases.  We show in Proposition \ref{prop:two} that if $\mathcal{C}$ is not symmetrically braided, $C_\mathcal{C}(\mathcal{C})$ is the maximal Tannakian subcategory of $\mathcal{C}$ (and hence $\mathcal{D}$) which is equal to $C_\mathcal{D}(\mathcal{D})$ or an index-$2$ fusion subcategory thereof.  Recall the relevant information from Sections \ref{sec:deeq} and \ref{sec:brad} pertaining to modular data, twists, and the equivariantization/de-equivariantization constructions.

\begin{lemma}\label{lem:one}
Let $\mathcal{C}$ be a braided near-integral fusion category and $\mathcal{E}$ a Tannakian subcategory.  Then $\mathcal{E}\subset C_\mathcal{C}(\mathcal{C})$.
\end{lemma}

\begin{proof}
We may assume by Lemma \ref{cor:1} that $\mathcal{C}$ is pseudounitary, equipped with its canonical spherical structure.  Let $\mathcal{E}\subset\mathcal{C}$ be any Tannakian subcategory.  We have $\theta_X=1$ for all $X\in\mathcal{O}(\mathcal{E})$ \cite[Proposition 3.11]{MR4655273}, hence by the balancing equation \cite[Proposition 8.13.8]{tcat},
\begin{equation}
S_{X,\rho}=\theta_X^{-1}\theta_\rho^{-1}\sum_{Y\in\mathcal{O}(\mathcal{C})}N_{X,\rho}^Y\dim(Y)\theta_Y=\dim(X)\dim(\rho).
\end{equation}
As $\mathcal{C}$ is pseudounitary, \cite[Proposition 2.5]{mug1} implies $\rho\in C_\mathcal{C}(\mathcal{E})$, hence $C_\mathcal{C}(\mathcal{E})=\mathcal{C}$ since $\rho$ $\otimes$-generates $\mathcal{C}$.  Moreover $\mathcal{E}\subset C_\mathcal{C}(\mathcal{C})$.
\end{proof}

\begin{proposition}\label{prop:two}
Let $\mathcal{C}$ be a braided near-integral fusion category which is not symmetrically braided.  Then $C_\mathcal{C}(\mathcal{C})$ is the maximal Tannakian subcategory of $\mathcal{C}$.
\end{proposition}

\begin{proof}
If $\mathcal{C}$ is not symmetrically braided, then $C_\mathcal{C}(\mathcal{C})\subset C_\mathcal{D}(\mathcal{D})$.  So if $C_\mathcal{D}(\mathcal{D})$ is Tannakian, then $C_\mathcal{C}(\mathcal{C})=C_\mathcal{D}(\mathcal{D})$ is Tannakian and the result follows from Lemma \ref{lem:one}.  If $C_\mathcal{D}(\mathcal{D})$ is not Tannakian, $C_\mathcal{D}(\mathcal{D})$ contains a maximal Tannakian subcategory $\mathcal{E}$ with $\dim(\mathcal{E})=(1/2)\dim(C_\mathcal{D}(\mathcal{D}))$ and as $\mathcal{E}\subset C_\mathcal{C}(\mathcal{C})$ by Lemma \ref{lem:one}, either $\mathcal{E}=C_\mathcal{C}(\mathcal{C})$ and we are done, or $C_\mathcal{C}(\mathcal{C})=C_\mathcal{D}(\mathcal{D})$ are both super-Tannakian.

\par To show the latter cannot occur, let $A\in\mathcal{E}$ be the regular algebra.  Note that the free $A$-module $\rho\otimes A$ is local and that $\mathcal{E}$ is the maximal Tannakian subcategory of both $\mathcal{C}$ and $\mathcal{D}$, so the de-equivariantizations $\mathcal{C}^0_A=\mathcal{C}_A$ and $\mathcal{D}^0_A=\mathcal{D}_A$ are both supermodular \cite[Sections 5.5--5.6]{DGNO}.  In particular, their Gauss sums \cite[Section 8.15]{tcat} vanish.  But we have
\begin{equation}
0=\tau_1^+(\mathcal{C}_A)=\tau_1^+(\mathcal{D}_A)+\theta_\rho\sum_{\substack{\text{simple }M\\M\subset\rho\otimes A}}\dim(M)^2=\theta_\rho\sum_{\substack{\text{simple }M\\M\subset\rho\otimes A}}\dim(M)^2,
\end{equation}
which cannot occur since $\theta_\rho$ is a root of unity and the free $A$-module $\rho\otimes A$ has at least one simple summand.
\end{proof}

\par We now demonstrate there is exactly one such braided near-integral fusion category $\mathcal{C}$ with $\mathrm{FPdim}(\mathcal{C})\not\in\mathbb{Z}$.

\begin{proposition}\label{prop:fib}
Let $\mathcal{C}$ be a braided near-integral fusion category.  If $\dim(\mathcal{C})\not\in\mathbb{Z}$, then $\mathcal{C}\simeq\mathcal{C}(A_1,5,q)_\mathrm{ad}$ is a braided equivalence where $q^2$ is a primitive $5$th root of unity.
\end{proposition}

\begin{proof}
The assumption $\dim(\mathcal{C})\not\in\mathbb{Z}$ implies $\kappa=kN$ for some $k\in\mathbb{Z}_{\geq1}$ by Lemma \ref{lem:divides} and $C_\mathcal{C}(\mathcal{C})=\mathcal{D}$ by \cite[Lemma 5.2]{MR4413278}.  The remainder follows the proof of \cite[Lemma 2]{MR4658217}.  Proposition \ref{prop:two} states that $C_\mathcal{C}(\mathcal{C})=\mathcal{D}$ is Tannakian with regular algebra $A$, assume $\mathcal{C}$ is pseudounitary without loss of generality by Lemma \ref{cor:1} and let $\sigma\in\mathrm{Gal}(\overline{\mathbb{Q}}/\mathbb{Q})$ be such that $\sigma(d_+)=d_-$.  We compute
\begin{equation}
\dim(\mathcal{C}_A^\sigma)=\sigma(\dim(\mathcal{C}_A))=1-\dfrac{d_-}{d+}=1+\dfrac{N}{d_+^2}=1+\dfrac{1}{1+kd_+}.
\end{equation}
We must have $(1+kd_+)^{-1}>1/3$ \cite[Theorem 4.1.1]{MR3943751}.  This implies $k=1$ since $d_+>1$, hence $d_+=(1/2)(N+\sqrt{N(N+4)})<2$ which implies $N=1$.  Moreover $\mathcal{C}$ is rank $2$ and the result follows \cite{ostrik}.
\end{proof}

 For the remainder of the section we will assume that $\mathrm{FPdim}(\mathcal{C})\in\mathbb{Z}$.

\begin{proposition}\label{prop:tan}
Let $\mathcal{C}$ be a nonsymmetrically braided near-integral fusion category with $\mathrm{FPdim}(\mathcal{C})\in\mathbb{Z}$.  If $C_\mathcal{D}(\mathcal{D})$ is Tannakian, then $C_\mathcal{C}(\mathcal{C})=\mathcal{D}$, and either
\begin{enumerate}
\item $\theta_\rho=\pm\zeta_4$, $\kappa=0$, $\dim(\mathcal{C})=2N$,
\item $\theta_\rho=\zeta_3^{\pm1}$, $N=2\kappa^2$, and $\dim(\mathcal{C})=3N=6\kappa^2$, or
\item $\theta_\rho=-1$, $N=(3/4)\kappa^2$, and $\dim(\mathcal{C})=4N=3\kappa^2$.
\end{enumerate}
In all cases, $d_+\in\mathbb{Z}$.
\end{proposition}

\begin{proof}
Let $A\in C_\mathcal{C}(\mathcal{C})$ be the regular algebra.  The categories $\mathcal{C}^0_A=\mathcal{C}_A$ and $\mathcal{D}^0_A=\mathcal{D}_A$ are both nondegenerately braided.  The simple summands of free $A$-modules on $X\in\mathcal{O}(\mathcal{D})$ and the free $A$-module $\rho\otimes A$ are disjoint. So with
\begin{equation}
M:=\sum_{\substack{\text{simple }Q\\Q\subset\rho\otimes A}}\dim(Q)^2,
\end{equation}
we compute by \cite[Proposition 8.15.4]{tcat}
\begin{align}
&&\dfrac{N}{\dim(A)}+M&=\dim(\mathcal{C}_A) \\
&&&=\tau^+_1(\mathcal{C}_A)\tau^-_1(\mathcal{C}_A) \\
&&&=\left(\tau_1^+(\mathcal{D}_A)+\theta_\rho M\right)\left(\tau_1^-(\mathcal{D}_A)+\theta^{-1}_\rho M\right) \\
&&&=\dfrac{N}{\dim(A)}+(\theta_\rho\tau_1^-(\mathcal{D}_A)+\theta^{-1}_\rho\tau_1^+(\mathcal{D}_A))M+M^2 \\
\Rightarrow&&0&=((\theta_\rho\tau_1^-(\mathcal{D}_A)+\theta^{-1}_\rho\tau_1^+(\mathcal{D}_A))-1+M)M.
\end{align}
We know $M\neq0$ since $\rho\otimes A$ has at least one simple summand.  Set $\xi:=\xi_1(\mathcal{D}_A)$ so that 
\begin{align}
M&=1-(\theta_\rho\tau_1^-(\mathcal{D}_A)+\theta^{-1}_\rho\tau_1^+(\mathcal{D}_A)) \\
&=1-((\theta^{-1}_\rho\xi)^{-1}+\theta^{-1}_\rho\xi)\sqrt{\dim(\mathcal{D}_A)}.
\end{align}
By assumption $M\in\mathbb{Z}$.   If $(\theta^{-1}_\rho\xi)^{-1}+\theta^{-1}_\rho\xi>0$, $M<1$ which cannot occur since $\rho\otimes A$ has at least one summand.  Moreover $\theta_\rho\xi^{-1}$ is a primitive root of unity of order $2$, $3$, or $4$.
\par If $\theta_\rho\xi^{-1}$ is a primitive $4$th root of unity, then $\theta_\rho\xi^{-1}+\theta^{-1}_\rho\xi=0$ and moreover $M=1$.  We compute in this case
\begin{align}
&&\dfrac{\dim(\mathcal{C})}{\dim(A)}&=\dfrac{N}{\dim(A)}+1 \\
\Rightarrow&&\dim(\mathcal{C})&=N+\dim(A)\leq2N
\end{align}
with equality if and only if $\kappa=0$.  Since $\dim(A)>0$ and $N$ divides $\dim(\mathcal{C})\in\mathbb{Z}$, then $\dim(A)=N$.  Therefore $\mathcal{C}$ is an equivariantization of an integral fusion category of rank $2$, so $\mathcal{C}$ is integral.

\par If $\theta_\rho\xi^{-1}$ is a primitive $3$rd root of unity, then $\theta_\rho\xi^{-1}+\theta^{-1}_\rho\xi=-1$ and therefore $M=1+\sqrt{\dim(\mathcal{D}_A)}$.  We compute in this case
\begin{align}
&&\dfrac{\dim(\mathcal{C})}{\dim(A)}&=\dfrac{N}{\dim(A)}+1+\sqrt{\dim(\mathcal{D}_A)} \\
\Rightarrow&&\dim(\mathcal{C})&=N+\dim(A)+\sqrt{N\dim(A)}\leq3N.
\end{align}
Since $N$ divides $\dim(\mathcal{C})\in\mathbb{Z}$ and $\dim(A)>0$, we must have $\dim(\mathcal{C})=3N$ or $\dim(\mathcal{C})=2N$.  In the latter case $\dim(A)+\sqrt{N\dim(A)}=N$ which implies $\dim(A)=N(1/2)(3\pm\sqrt{5})$ which is not an integer.  We may then conclude $\dim(\mathcal{C})=3N$ and $\dim(A)=N$.

\par Lastly if $\theta_\rho\xi^{-1}=-1$, then $M=1+2\sqrt{\dim(\mathcal{D}_A)}$.   We compute in this case
\begin{align}
&&\dfrac{\dim(\mathcal{C})}{\dim(A)}&=\dfrac{N}{\dim(A)}+1+2\sqrt{\dim(\mathcal{D}_A)} \\
\Rightarrow&&\dim(\mathcal{C})&=N+\dim(A)+2\sqrt{N\dim(A)}\leq4N.
\end{align}
Since $N$ divides $\dim(\mathcal{C})\in\mathbb{Z}$ and $\dim(A)>0$, we must have $\dim(\mathcal{C})=4N$, $\dim(\mathcal{C})=3N$ or $\dim(\mathcal{C})=2N$.  In the second case, $\dim(A)+2\sqrt{N\dim(A)}=2N$ which implies $\dim(A)=2N(2\pm\sqrt{3})$ which is not an integer.  In the third case, $\dim(A)+2\sqrt{N\dim(A)}=N$ which implies $\dim(A)=N(3\pm2\sqrt{2})$ which is not an integer.  We may then conclude $\dim(\mathcal{C})=4N$ and $\dim(A)=N$.
\end{proof}

\begin{proposition}\label{prop:stan}
Let $\mathcal{C}$ be a braided near-integral fusion category with $\mathrm{FPdim}(\mathcal{C})\in\mathbb{Z}$.  If $C_\mathcal{D}(\mathcal{D})$ is super-Tannakian, then $\mathcal{D}=C_\mathcal{D}(\mathcal{D})$, $\kappa=0$ and $\theta_\rho^{16}=1$.
\end{proposition}

\begin{proof}
Let $A\in\mathcal{D}$ the regular algebra of $C_\mathcal{C}(\mathcal{C})$.  Then $\mathcal{C}^0_A=\mathcal{C}_A$ is modular and $\mathcal{D}^0_A=\mathcal{D}_A$ is supermodular, thus $\tau^+_1(\mathcal{D}_A)=0$ and $\tau^+_1(\mathcal{C}_A)=\xi_1(\mathcal{C}_A)\sqrt{\dim(\mathcal{C}_A)}$ (by definition).  The simple summands of the free modules on $X\in\mathcal{O}(\mathcal{D})$ and the simple summands of $\rho\otimes A$ are disjoint, so we compute
\begin{align}
\xi_1(\mathcal{C}_A)\sqrt{\dim(\mathcal{C}_A)}=\tau^+_1(\mathcal{C}_A)&=\tau^+_1(\mathcal{D}_A)+\theta_\rho \sum_{\substack{\text{simple }M\\M\subset\rho\otimes A}}\dim(M)^2 \\
&=\theta_\rho\sum_{\substack{\text{simple }M\\M\subset\rho\otimes A}}\dim(M)^2,
\end{align}
therefore $\theta_\rho=\xi_1(\mathcal{C}_A)$ must be a $16$th root of unity by \cite[Proposition 2.6]{2019arXiv191212260G} since $\mathcal{C}$ is weakly integral.   Now we have also that
\begin{align}
\dim(\mathcal{C}_A)&=\left(\sum_{\substack{\text{simple }M\\M\subset\rho\otimes A}}\dim(M)^2\right)^2=\left(\dim(\mathcal{C}_A)-\dim(\mathcal{D}_A)\right)^2.\label{foorty}
\end{align}
Hence by the quadratic formula,
\begin{align}
\dim(\mathcal{C}_A)&=\dim(\mathcal{D}_A)+\dfrac{1}{2}\left(1+\sqrt{1+4\dim(\mathcal{D}_A)}\right) \\
&\leq2\dim(\mathcal{D}_A)
\end{align}
since $\dim(\mathcal{D}_A)$ is an integer greater than or equal to $2$ as $\mathcal{D}_A$ is supermodular.  Moreover $\dim(\mathcal{C})=2\dim(\mathcal{D})$ because $\dim(\mathcal{C})/\dim(\mathcal{D})\in\mathbb{Z}_{\geq2}$.  So we have $\kappa=0$.  Equation (\ref{foorty}) then implies
\begin{align}
\dfrac{2\dim(\mathcal{D})}{\dim(A)}=\left(\dfrac{2\dim(\mathcal{D})}{\dim(A)}-\dfrac{\dim(\mathcal{D})}{\dim(A)}\right)^2=\dfrac{\dim(\mathcal{D})^2}{\dim(A)^2},
\end{align}
and moreover $\dim(\mathcal{D})=2\dim(A)$, i.e.\ $\mathcal{D}=C_\mathcal{D}(\mathcal{D})$ is symmetrically braided.
\end{proof}

\begin{proposition}\label{prop:rep}
Let $\mathcal{C}$ be a braided near-integral fusion category.  If $\dim(\mathcal{C})=3N$ for some $N\in\mathbb{Z}_{\geq1}$, then $\mathcal{C}$ is equivalent as a fusion category to $\mathrm{Rep}(G)$ for a finite group $G$.
\end{proposition}

\begin{proof}
If $\mathcal{C}$ is symmetrically braided, we are done, otherwise Proposition \ref{prop:tan} implies $\mathcal{C}$ is an equivariantization $\mathcal{C}(C_3,q)^A$ for a nondegenerate quadratic form $q:C_3\to\mathbb{C}^\times$ and finite group $A$.  Therefore we have an exact sequence of fusion categories, in the sense of \cite{MR2863377},
\begin{equation}
\mathrm{Rep}(A)\to\mathcal{C}\stackrel{F}{\to}\mathcal{C}(C_3,q)
\end{equation}
where $F$ is the modularization functor.  Lemma 7.4 of \cite{MR2863377} then implies $\mathcal{C}(C_3,q)$ admits a fiber functor, and moreover so does $\mathcal{C}$.  Therefore $\mathcal{C}$ is equivalent as a fusion category to the category of representations of a semisimple quasitriangular Hopf algebra $H$.  By \cite[Proposition 4.2]{MR3012244}, $H$ is twist equivalent to a Hopf algebra $H'$ that fits into a cocentral exact sequence
\begin{equation}
\mathbb{C}\to\mathbb{C}C_3\to H'\to\mathbb{C}A\to\mathbb{C}.
\end{equation}
We then have the category $\mathrm{coRep}(H')$ of corepresentations of $H'$ has an $A$-grading such that the trivial component is the category $\mathrm{coRep}(C_3)$ of corepresentations of $C_3$ \cite[Proposition 4.1]{MR3012244}.  Therefore $\mathrm{coRep}(H')$ is pointed and moreover $H'$ is a group algebra.
\end{proof}

\begin{example}  The non-symmetrically braided fusion categories with the fusion rules of $\mathrm{Rep}(S_3)$ can be found in the double $\mathcal{Z}(\mathrm{Rep}(S_3))$ by generating a subcategory with any simple object of dimension $2$ and twist which is a primitive third root of unity (see Example \ref{braidfin}).  This example can viewed this as the $n=0$ element in the infinite series $3_-^{1+2n}\rtimes C_2$, though by definition $C_3$ is not considered extra-special.

\par It is elementary to verify that in the untwisted double $\mathcal{Z}(\mathrm{Vec}_G)$ for $G\cong 3_-^{1+2n}\rtimes C_2$ any simple object of dimension $2\cdot3^n$ whose twist is a primitive third root of unity generates a braided near-integral subcategory with Tannakian $\mathcal{D}=C_\mathcal{D}(\mathcal{D})$.  This subcategory has the same fusion rules as $\mathrm{Rep}(3_-^{1+2n}\rtimes C_2)$ but is not symmetrically braided.  The same construction does not generalize to extra-special $p$-groups of exponent $p$.
\end{example}


\section{Classification of premodular fusion categories}\label{sec:class}

\par The clear method for classifying premodular fusion categories (see Section \ref{subsec:braid}) of a fixed rank is to argue based on how symmetric the braiding is.  Many tools exist at the extreme ends of this spectrum.  For example, if the braiding is nondegenerate, then such categories have associated modular data \cite[Section 8.17]{tcat} which makes classification up to Grothendieck equivalence a mostly computational task.  Finite lists of possible modular data for modular fusion categories up to rank 11 already appears in the literature \cite{ng2023classification} with higher ranks surely to follow using identical techniques.  On the opposite end of the spectrum, the classification of symmetrically braided fusion categories lies in the realm of the representation theory of finite groups.  Each symmetrically braided fusion category is braided equivalent to one of the form $\mathrm{Rep}(G,\nu)$ where $G$ is a finite group and $\nu\in G$ is a central element of order at most $2$ describing a potential twisting of the braiding from the standard one \cite[Corollary 9.9.25]{tcat}; if $\nu=e$ is trivial, then we will often abbreviate $\mathrm{Rep}(G):=\mathrm{Rep}(G,e)$.  For other one-off notation, we denote the rank $1$ symmetrically braided fusion category $\mathrm{Vec}=\mathrm{Rep}(C_1,e)$ and the non-Tannakian rank $2$ symmetrically braided fusion category as $\mathrm{sVec}:=\mathrm{Rep}(C_2,\nu)$ for $\nu\neq e$.   Finite groups with less than or equal to $14$ isomorphism classes of irreducible representations have been explicitly described and conveniently cataloged in the literature (see \cite{MR2308490} and references within).  In Figure \ref{fig:symm} we collect the finite groups with less than or equal to $6$ conjugacy classes as they are the ones relevant to our current work, and count the number of central elements $\nu$ with $\nu^2=e$.  We use standard notation for finite groups including cyclic groups of order $n$ ($C_n$), dihedral groups of order $2n$ ($D_n$), symmetric groups on $n$ elements ($S_n$), alternating groups on $n$ elements ($A_n$), the quaternion group of order $8$ ($Q_8$), the group of affine linear transformations of the finite field of order $n$ ($F_n$), the dicyclic groups of order $4n$ ($\mathrm{Dic}_n$), and the projective special unitary groups ($\mathrm{PSU}(n,m)$).  The nontrivial semidirect products $\rtimes$ are unambiguous.

\begin{figure}[H]
\centering
\begin{equation*}
\begin{array}{|c|c|c|}
\hline G & |G| & \#\nu \\\hline\hline
C_1 & 1 & 1 \\\hline
C_2 & 2 & 2 \\\hline
C_3 & 3 & 1 \\
S_3 & 6 & 1  \\\hline
C_4 & 4 & 2 \\
C_2^2 & 4 & 4 \\
D_5 & 10 & 1 \\
A_4 & 12 & 1 \\\hline
\end{array}\,\,
\begin{array}{|c|c|c|}
\hline G & |G| & \#\nu \\\hline\hline
C_5 & 5 & 1 \\
D_4 & 8 & 2 \\
Q_8 & 8 & 2 \\
D_7 & 14 & 1 \\
F_5 & 20 & 1 \\
C_7\rtimes C_3 & 21 & 1 \\
S_4 & 24 & 1 \\
A_5 & 60 & 1 \\\hline
\end{array}\,\,
\begin{array}{|c|c|c|}
\hline G & |G| & \#\nu \\\hline\hline
C_6 & 6 & 2 \\
D_6 & 12 & 2 \\
\mathrm{Dic}_3 & 12 & 2 \\
D_9 & 18 & 1 \\
C_3\rtimes S_3 & 18 & 1 \\
C_3^2\rtimes C_4 & 36 & 1 \\
\mathrm{PSU}(3,2) & 72 & 1 \\
\mathrm{PSU}(2,7) & 168 & 1 \\\hline
\end{array}
\end{equation*}
    \caption{Finite groups with $6$ or less conjugacy classes}%
    \label{fig:symm}%
\end{figure}

\par The braided fusion categories inbetween nondegenerately braided and symmetrically braided have fewer tools available.  Some progress has been made in one case near to nondegenerately braided: when the symmetric center is braided equivalent to $\mathrm{sVec}$.  A complete classification of the possible premodular data for these so-called \emph{supermodular} fusion categories up to rank 6 can be found in \cite{MR4109138}, extrapolating tools used in the nondegenerate case.  This still leaves the case when the symmetric center is $\mathrm{Rep}(C_2)$, which one could argue is equally symmetric.  The results of Section \ref{seccat} on near-integral fusion can be seen as probing the nonsymmetric-but-degenerately braided world from the opposite side by considering those premodular fusion categories which are almost symmetric, but not quite.  We have shown that a braided fusion category $\mathcal{C}$ has a proper fusion subcategory $\mathcal{D}\subset\mathcal{C}$ of maximal rank if and only if $\mathcal{D}$ is symmetrically braided, and we have provided a highly-restrictive classification of such categories in Propositions \ref{prop:tan} and \ref{prop:stan}.

\par We note that our classification of premodular fusion categories of small rank below is up to braided equivalence; we make no effort to characterize the possible pivotal/spherical structures on these braided categories though all but one family \cite{MR4327964} have a canonical pivotal structure such that the categorical and Frobenius-Perron dimensions of simple objects coincide.


\subsection{Premodular categories of rank $\leq4$}\label{subsec:4}

\par Here we reiterate the known classification of premodular fusion categories of rank $3$ or less, and improve the known results on premodular fusion categories of rank $4$ which were previously characterized only by their fusion rules in \cite[Theorem 4.11]{MR3548123}.  The braided equivalence classes of premodular fusion categories of rank $3$ or less are listed in Figure \ref{fig:catranklessthan4} and those of rank $4$ are listed in Figure \ref{fig:catrank4}.  In these tables $q$ is such that $q^2$ is a primitive root of unity of order $m\in\mathbb{Z}_{\geq2}$ and $[n]_m:=(q^n-q^{-n})/(q-q^{-1})$ for $n\in\mathbb{Z}$.

\begin{figure}[H]
\centering
\begin{align*}
\begin{array}{|c|c|c|c|c|}
\hline \mathcal{C} & \mathrm{FPdim}(\mathcal{C}) & \mathrm{FPdims} & C_\mathcal{C}(\mathcal{C}) & \# \\\hline\hline
\mathrm{Vec} & 1 & 1 & \mathrm{Vec} & 1 \\\hline\hline
\mathrm{Rep}(C_2,\nu) & 2 & 1,1 & \mathrm{Rep}(C_2,\nu) & 2 \\
\mathcal{C}(C_2,q) & 2 & 1,1 & \mathrm{Vec} & 2 \\
\mathcal{C}(A_1,5,q)_\mathrm{ad} & \frac{5}{4}\csc^2(\frac{\pi}{5}) & 1,[3]_5 & \mathrm{Vec} & 4 \\\hline\hline
\mathrm{Rep}(C_3) & 3 & 1,1,1 & \mathrm{Rep}(C_3) & 1 \\
\mathcal{C}(C_3,q) & 3 & 1,1,1 & \mathrm{Vec} & 2\\
\mathcal{I}_q & 4 & 1,1,\sqrt{2} & \mathrm{Vec} & 8\\
\mathrm{Rep}(S_3) & 6 & 1,1,2 & \mathrm{Rep}(S_3) & 1 \\
\mathrm{Rep}(S_3)^\omega & 6 & 1,1,2 & \mathrm{Rep}(C_2) & 2\\
\mathcal{C}(A_1,7,q)_\mathrm{ad} & \frac{7}{4}\csc^2(\frac{\pi}{7}) &1,[3]_7,[5]_7 & \mathrm{Vec} & 6 \\
\hline
\end{array}
\end{align*}
    \caption{The $29$ braided equivalence classes of premodular fusion categories of rank $\leq3$, separated by rank from least (top) to greatest (bottom)}%
    \label{fig:catranklessthan4}%
\end{figure}

\par Premodular fusion categories of rank $1$ must be symmetrically braided so there is a unique class up to braided equivalence by Figure \ref{fig:symm}.  Premodular fusion categories of rank $2$, if not symmetrically braided, are modular.  A modular fusion category of rank $2$ has either a transitive Galois action \cite{MR4389082} or the Galois action on simple objects is trivial, hence $\mathcal{C}$ is integral and moreover pointed (see Example \ref{ex:braidgroup}).  This produces eight premodular fusion categories of rank $2$ up to braided equivalence: four are pointed, with two symmetric and two modular, and four are the Galois conjugate transitive modular fusion categories $\mathcal{C}(A_1,5,q)_\mathrm{ad}$ where $q^2$ is a primitive fifth root of unity \cite[Theorem I]{MR4389082}.  As a result, all are Galois conjugate to a pseudounitary premodular fusion category so there exists a canonical spherical structure via \cite[Section 9.5]{tcat}.

\par If a premodular category of rank $3$ is not symmetric or modular, then its symmetric center is a maximal fusion subcategory of rank/dimension $2$ and the results of Section \ref{sec:brad} apply.  If $\kappa=0$ then $\mathcal{C}$ would be an Ising braided fusion category, hence it is modular, which violates having a nontrivial symmetric center.  The only other possibility is that $\kappa=1$ by Propositions \ref{prop:tan} and \ref{prop:stan}, and moreover by Proposition \ref{prop:rep}, $\mathcal{C}$ is equivalent as a fusion category to $\mathrm{Rep}(S_3)$.  Braidings on such a near-group fusion category are given in Example \ref{braidfin}; there are two nonsymmetrically braided equivalence classes of such categories which we label $\mathrm{Rep}(S_3)^\omega$ as they are $C_2$-equivariantizations (see Section \ref{sec:deeq}) of pointed premodular categories $\mathcal{C}(C_3,q)$ where $q:C_3\to\mathbb{C}^\times$ takes value $\omega$, a primitive third root of unity, on the nontrivial elements.  This leaves the symmetrically braided fusion categories of rank $3$, of which there are two braided equivalence classes as seen in Figure \ref{fig:symm}, and the modular fusion categories of rank $3$, of which there are sixteen braided equivalence classes.  By inspecting the 24 possible modular data of rank $3$ \cite[Appendix E.2]{ng2023classification}, six of these have global dimension which are cubic integers, hence are the Galois conjugate transitive modular fusion categories $\mathcal{C}(A_1,7,q)_\mathrm{ad}$ where $q^2$ is a primitive $7$th root of unity \cite[Theorem I]{MR4389082}.  If $\mathcal{C}$ is weakly integral, then it is either (1) pointed, giving the two braided equivalence classes of rank $3$ pointed modular fusion categories, (2) it is strictly weakly integral, and therefore has a maximal pointed fusion subcategory, which forces $\mathcal{C}$ to be an Ising modular fusion category, providing another $8$ braided equivalence classes \cite[Appendix B.3]{DGNO}, or (3) $\mathcal{C}$ is integral and not pointed; there are no categories of the latter type as they must have a maximal pointed fusion subcategory \cite[Lemma 4.2]{schopieray2023fixedpointfree} and therefore are not modular by our discussion above.  This accounts for all $24$ possible modular data since each of the $8$ braided equivalence classes of Ising braided fusion categories has exactly two inequivalent spherical structures \cite[Appendix B.5]{DGNO}.

\begin{note}
Observe that we have shown that all premodular fusion categories of rank $3$ or less are \emph{multiplicity-free}, i.e.\ none of the fusion coefficients exceed $1$.  As a result, all of the categories in Figure \ref{fig:catranklessthan4} are contained in Table 9.1 of \cite{vercleyen2024lowrank} which lists all multiplicity-free fusion categories of ranks $\leq7$ (see also \cite{MR4644689}).
\end{note}

\begin{figure}[H]
\centering
\begin{align*}
\begin{array}{|c|c|c|c|c|}
\hline \mathcal{C} & \mathrm{FPdim}(\mathcal{C}) & \mathrm{FPdims} & C_\mathcal{C}(\mathcal{C}) & \# \\\hline\hline
\mathrm{Rep}(C_4,\nu) & 4 & 1,1,1,1 & \mathrm{Rep}(C_4,\nu) & 2 \\
\mathrm{Rep}(C_2^2,\nu) & 4 & 1,1,1,1 & \mathrm{Rep}(C_2^2,\nu) & 2 \\
\mathrm{Rep}(D_5) & 10 & 1,1,2,2 & \mathrm{Rep}(D_5) & 1 \\
\mathrm{Rep}(A_4) & 12 & 1,1,1,3 & \mathrm{Rep}(A_4) & 1\\
\hline\hline
\mathrm{Rep}(A_4)^\epsilon & 12 & 1,1,1,3 & \mathrm{Rep}(C_3) & 1 \\
\hline\hline
\mathcal{C}(C_2,q)\boxtimes\mathrm{Rep}(C_2,\nu) & 4 & 1,1,1,1 & \mathrm{Rep}(C_2,\nu) & 3\\
\mathcal{C}(C_4,q_{\pm i}) & 4 & 1,1,1,1 & \mathrm{Rep}(C_2) & 2 \\ 
\mathrm{Rep}(D_5)^\mu & 10 & 1,1,2,2 & \mathrm{Rep}(C_2) & 2 \\
\mathcal{C}(A_1,5,q)_\mathrm{ad}\boxtimes\mathrm{Rep}(C_2,\nu) & \frac{5}{2}\csc^2(\frac{\pi}{5}) & 1,1,[3]_5,[3]_5 & \mathrm{Rep}(C_2,\nu) & 8 \\
\mathcal{C}(A_1,8,q)_\mathrm{ad} & 2\csc^2(\frac{\pi}{8}) & 1,1,[3]_8,[3]_8 & \mathrm{sVec} & 2 \\
\hline\hline
\mathcal{C}(C_4,q) & 4 & 1,1,1,1 & \mathrm{Vec} & 4 \\
\mathcal{C}(C_2^2,q) &4 & 1,1,1,1 & \mathrm{Vec} & 5 \\
\mathcal{C}(A_1,5,q_1)_\mathrm{ad}\boxtimes\mathcal{C}(C_2,q_2) & \frac{5}{2}\csc^2(\frac{\pi}{5}) & 1,1,[3]_5,[3]_5 & \mathrm{Vec} & 8  \\
\mathcal{C}(A_1,5,q_1)_\mathrm{ad}\boxtimes\mathcal{C}(A_1,5,q_2)_\mathrm{ad} & 5[3]_5^2 & 1,[3]_5,[3]_5,[3]_5^2 & \mathrm{Vec} & 10 \\
\mathcal{C}(A_1,9,q)_\mathrm{ad} & \frac{9}{4}\csc^2(\frac{\pi}{9}) & 1,[3]_9,[5]_9,[7]_9 & \mathrm{Vec} & 6 \\
\hline
\end{array}
\end{align*}
    \caption{The $57$ braided equivalence classes of premodular fusion categories of rank $4$, separated by rank of symmetric center from greatest (top) to least (bottom)}%
    \label{fig:catrank4}%
\end{figure}

\par There are six braided equivalence classes of symmetrically braided fusion categories of rank $4$ as seen in Figure \ref{fig:symm}; though $C_2^2$ has three nontrivial central elements of order $2$, these three braided fusion categories are related by braided autoequivalences permuting the nontrivial elements (see \cite[Section 8.4]{tcat}).  If $\mathcal{C}$ is a premodular fusion category of rank $4$ whose symmetric center is rank $3$, the symmetric center is Tannakian by Figure \ref{fig:symm} and Proposition \ref{prop:tan} implies $\mathcal{C}$ is a near-group braided fusion category with the fusion rules of $\mathrm{Rep}(A_4)$; there is a unique braided equivalence class of nonsymmetrically braided fusion categories of this type (see Example \ref{braidfin}) which is equivalent as a fusion category to $\mathrm{Rep}(A_4)$ so we simply decorate $\mathrm{Rep}(A_4)^\epsilon$ with a primitive second root of unity $\epsilon$ analogously to the $\mathrm{Rep}(S_3)^\omega$ notation in the rank $3$ classification.

\par A premodular fusion category of rank $4$ with a symmetric center of rank $2$ has a symmetric center which is braided equivalent to $\mathrm{Rep}(C_2,\nu)$.  The fusion rules in the case of $\mathrm{sVec}$ were described in \cite{MR4109138}; if the category does not factor, then there is a unique set of fusion rules, whose categorifications as fusion categories are classified in \cite[Theorem A.3]{MR4486913}.  In particular $\mathcal{C}$ is equivalent to $\mathcal{C}(A_1,8,q)_\mathrm{ad}$ where $q^2$ is a primitive $8$th root of unity; there are two braided equivalence classes of such categories \cite[Pg. 151]{vercleyen2024lowrank}.

\par If $C_\mathcal{C}(\mathcal{C})\simeq\mathrm{Rep}(C_2)$, then $C_\mathcal{C}(\mathcal{C})$ is a maximal Tannakian fusion subcategory measured by dimension, thus $\mathcal{C}_A$ is modular where $A$ is the regular algebra of $\mathrm{Rep}(C_2)$ (see Section \ref{sec:deeq}).  When the two invertible objects act on the remaining two isomorphism classes of simple objects transitively, then $\mathcal{C}$ is a braided generalized near-group fusion category; those with irrational dimension were classified in \cite{MR4658217} and are products of $\mathrm{Rep}(C_2)$ and $\mathcal{C}(A_1,5,q)_\mathrm{ad}$ where $q^2$ is any primitive fifth root of unity.  In the case $\mathcal{C}$ is weakly integral, the category $\mathcal{C}_A$ is an integral modular fusion category of rank $2$; the $C_2$-action on $\mathcal{C}_A$ is trivial, hence $\mathcal{C}$ is integral of Frobenius-Perron dimension $4$, and must be pointed.  Thus either $\mathcal{C}$ factors or $\mathcal{C}\simeq\mathcal{C}(C_4,q_{\pm i})$ is a braided equivalence where $q_{\pm i}:C_4\to\mathbb{C}^\times$ takes the value $\pm i$ on a generator and $1$ on the element of order $2$.   In the case the symmetric center acts trivially on the remaining isomorphism classes of simple objects, $\mathcal{C}_A$ is modular of rank $5$ with two pairs of simple objects of equal dimensions.  There are only two such possible modular data which are both pointed of rank 5 \cite[Appendix E.4]{ng2023classification}, hence the dimensions of the nontrivial simple objects of $\mathcal{C}$ are both $2$ and these objects are self-dual as they have distinct twists.  Since there is a unique pointed fusion category of rank $5$ which is braided, $\mathrm{Rep}(C_5)$, and a unique fixed-point free $C_2$-action on $\mathrm{Rep}(C_5)$, then there is a unique fusion category, $\mathrm{Rep}(D_5)$, up to equivalence with simple objects of dimension $1,1,2,2$ which is braided.  But for each choice of the three braidings on $\mathrm{Rep}(C_5)$, the fixed-point free autoequivalence is braided, giving three distinct braided fusion categories which we denote by $\mathrm{Rep}(D_5)^\mu$ where $\mu=\zeta_5^a$ where $a$ is either a quadratic residue or quadratic nonresidue modulo $5$.


\subsection{Premodular categories of rank $5$}

A classification of premodular categories of rank $5$ was purported in \cite[Theorem I.1]{MR3743161} which is correct up to Grothendieck equivalence.  But categories are missing from the main theorem and one can be more precise with the remainder of the classification.  The absent categories are those braided Tambara-Yamagami fusion categories with symmetric center braided equivalent to $\mathrm{Rep}(C_2^2)$; the four braided equivalence classes of this type were collected in \cite[Figure 4]{argentina} and are equivalent as fusion categories to either $\mathrm{Rep}(D_4)$ or $\mathrm{Rep}(Q_8)$.  Our argument is based on the results from Section \ref{seccat} and results which have appeared since \cite{MR3743161} was published.  Our argument is organized by the rank of the symmetric center for ease of reading, with rank $4$ and $5$ Tannakian subcategories described in Figure \ref{fig:catrank5} and ranks less than or equal to $3$ described in Figure \ref{fig:catrank5b}.

\subsubsection{Symmetric center of rank $4$ or $5$}

\begin{figure}[H]
\centering
\begin{align*}
\begin{array}{|c|c|c|c|c|}
\hline \mathcal{C} & \mathrm{FPdim}(\mathcal{C}) & \mathrm{FPdims} & C_\mathcal{C}(\mathcal{C}) & \# \\\hline\hline
\mathrm{Rep}(C_5) & 5 & 1,1,1,1,1 & \mathrm{Rep}(C_5) & 1\\
\mathrm{Rep}(D_4,\nu) & 8 & 1,1,1,1,2 & \mathrm{Rep}(D_4,\nu) & 2 \\
\mathrm{Rep}(Q_8,\nu) & 8 & 1,1,1,1,2 & \mathrm{Rep}(Q_8,\nu) & 2  \\
\mathrm{Rep}(D_7) & 14 & 1,1,2,2,2 & \mathrm{Rep}(D_7) & 1  \\
\mathrm{Rep}(F_5) & 20 & 1,1,1,1,4 & \mathrm{Rep}(F_5) & 1  \\
\mathrm{Rep}(C_7\rtimes C_3) & 21 & 1,1,1,3,3 & \mathrm{Rep}(C_7\rtimes C_3) & 1 \\
\mathrm{Rep}(S_4) & 24 & 1,1,2,3,3 & \mathrm{Rep}(S_4) & 1  \\
\mathrm{Rep}(A_5) & 60 & 1,3,3,4,5 & \mathrm{Rep}(A_5) & 1  \\
\hline\hline
\mathrm{Rep}(D_4)^{\pm i} & 8 & 1,1,1,1,2 & \mathrm{Rep}(C_2^2)  & 2 \\
\mathrm{Rep}(Q_8)^{\pm i} & 8 &  1,1,1,1,2 & \mathrm{Rep}(C_2^2) & 2 \\
\hline
\end{array}
\end{align*}
    \caption{The $14$ braided equivalence classes of rank $5$ braided fusion categories with Tannakian subcategory of maximal rank $5$ (above) or $4$ (below)}%
    \label{fig:catrank5}%
\end{figure}

The case when $\mathcal{C}$ is symmetrically braided is described by Figure \ref{fig:symm}, so we assume $C_\mathcal{C}(\mathcal{C})$ is a maximal fusion subcategory; in particular, $\mathcal{C}=\mathcal{C}(C_\mathcal{C}(\mathcal{C}),\kappa)$ is a non-symmetrically braided near-integral fusion category for some $\kappa\in\mathbb{Z}_{\geq0}$.  Evidently $\mathcal{C}$ has rank larger than $2$, so Proposition \ref{prop:fib} implies $\dim(\mathcal{C})\in\mathbb{Z}$. Propositions \ref{prop:tan} and \ref{prop:stan} then state the only options for $\kappa\in\mathbb{Z}$ are $0$, $\sqrt{N/2}$ or $\sqrt{4N/3}$.  As $N$ is restricted to the orders in the first column of Figure \ref{fig:symm}, the latter roots are integers only in one case: $N=12$, $\kappa=4$, $d_+=6$, $\dim(\mathcal{C})=48$, and $\theta_\rho=-1$.  The de-equivariantization by the three invertible objects produces a category whose invertible objects form the group $C_2^2$, and $3$ simple objects of dimension $2$, with a fixed-point-free action of $C_3$ by tensor autoequivalences.  It was shown in \cite[Proposition 6.1]{schopieray2023fixedpointfree} that such a category cannot exist.  Lastly, if $\kappa=0$ then $N$ must be a perfect square when $C_\mathcal{C}(\mathcal{C})$ is Tannakian by Proposition \ref{prop:tan} as $\mathcal{C}$ is integral, which only occurs if $\mathcal{C}$ is a near-group braided fusion category, or more specifically, a braided Tambara-Yamagami fusion category whose twist on the noninvertible element is $\pm i$.  There are $4$ inequivalent categories of this type described in detail in \cite[Figure 4]{argentina}.  Otherwise $C_\mathcal{C}(\mathcal{C})$ is super Tannakian, and thus pointed, and again $\mathcal{C}$ is a Tambara-Yamagami braided fusion category.  But the symmetric centers of such categories are rank $2$ so they do not belong to this case.

\subsubsection{Symmetric center of rank $3$ or less}

\begin{figure}[H]
\centering
\begin{align*}
\begin{array}{|c|c|c|c|c|}
\hline \mathcal{C} & \mathrm{FPdim}(\mathcal{C}) & \mathrm{FPdims} & C_\mathcal{C}(\mathcal{C}) & \# \\\hline\hline
\mathcal{C}(C_2^2,q_-)^{S_3} & 24 & 1,1,2,3,3 & \mathrm{Rep}(S_3) & 1  \\\hline\hline
(\mathcal{I}_{q_1}\boxtimes\mathcal{I}_{q_2})_\mathbb{Q} & 8 & 1,1,1,1,2 & \mathrm{Rep}(C_2) & 12 \\
\mathrm{Rep}(D_7)^\psi & 14 & 1,1,2,2,2 & \mathrm{Rep}(C_2) & 2  \\
\mathcal{C}(A_1,10,q)_\mathrm{ad} & 10[3]_5^2 & 1,1,2[3]_5,[3]_5^2,[3]_5^2 & \mathrm{Rep}(C_2) & 4  \\\hline\hline
\mathcal{C}(C_5,q) & 5 & 1,1,1,1,1 & \mathrm{Vec} & 2 \\
(\mathrm{Rep}(S_3)^\omega)^\gamma & 12 & 1,1,2,\sqrt{3},\sqrt{3} & \mathrm{Vec} & 4 \\
\mathcal{C}(A_1,11,q)_\mathrm{ad} & \frac{11}{4}\csc^2(\pi/11) & 1,[2]_{11},[4]_{11}, & \mathrm{Vec} & 10 \\
& &[6]_{11},[8]_{11} & & \\
\mathcal{C}(A_2,7,q)_\mathrm{ad} &\frac{7^2}{2^8}\csc^6(\frac{\pi}{7})\sec^2(\frac{\pi}{7}) & 1,\frac{[4]_7[5]_7}{[2]_7},\frac{[4]_7[5]_7}{[2]_7}, & \mathrm{Vec} & 6 \\
& & [2]_7[4]_7,\frac{[3]_7^2[6]_7}{[2]_7}& & \\
\hline
\end{array}
\end{align*}
    \caption{The $41$ braided equivalence classes of premodular fusion categories of rank $5$ with rank $\leq3$ maximal Tannakian subcategory}%
    \label{fig:catrank5b}%
\end{figure}


\paragraph{Rank $3$ symmetric center}

Assume first that the symmetric center is braided equivalent to $\mathrm{Rep}(C_3)$.  Then $\mathcal{O}(\mathrm{Rep}(C_3))$ acts trivially on the two remaining isomorphism classes of simple objects by the orbit-stabilizer theorem.  We must have $C_\mathcal{C}(\mathcal{C})\simeq\mathrm{Rep}(C_3)$ is maximal Tannakian since any maximal subcategory which is Tannakian would necessarily be contained in the center by the balancing equation \cite[Proposition 8.13.8]{tcat}.  Moreover with $A$ the regular algebra of $\mathrm{Rep}(C_3)$, $\mathcal{C}_A$ is a modular fusion category of rank $7$ with two pairs of three isomorphism classes of simple objects of the same dimension and the same twists.  There is no such modular data \cite[Appendix E.6]{ng2023classification}, so this case produces no examples.  Otherwise $C_\mathcal{C}(\mathcal{C})$ is braided equivalent to $\mathrm{Rep}(S_3)$.  Let $\rho_1,\rho_2\in\mathcal{O}(\mathcal{C})\setminus\mathcal{O}(\mathrm{Rep}(S_3))$ be the other isomorphism classes of simple objects.  The rank of $\mathcal{C}_A$ where $A$ is the regular algebra of $\mathrm{Rep}(C_2)\subset\mathrm{Rep}(S_3)$ is rank $4$ or $7$, depending on whether the nontrivial invertible element of $\mathrm{Rep}(S_3)$ acts trivially or transitively on $\mathcal{O}(\mathcal{C})\setminus\mathcal{O}(\mathrm{Rep}(S_3))$.  In either case $C_\mathcal{C}(\mathcal{C})$ is maximal Tannakian by Figure \ref{fig:symm}.  Then $\mathcal{C}_A$ is modular and by Lemma 10 below, if $\mathcal{O}(\mathrm{Rep}(S_3))$ does not act trivially on $\mathcal{O}(\mathcal{C})\setminus\mathcal{O}(\mathrm{Rep}(S_3))$, then $\mathcal{C}$ is integral of dimension $6$, $12$, $18$ or $24$.  The only possible dimensions for the nontrivial simple objects of $\mathcal{C}$ are then $1,1,2,3,3$ by a brute-force check, hence $\mathcal{C}$ is an $S_3$-equivariantization of $\mathcal{C}(C_2^2,q_-)$ where $q_-$ takes the value $-1$ on all nontrivial elements of $C_2^2$, and the $S_3$-action on $\mathcal{C}(C_2^2,q_-)$ by braided autoequivalences is transitive on the nontrivial invertibles.  Such a category is unique up to braided equivalence and while the fusion rules are those of $\mathrm{Rep}(S_4)$ but we emphasize that $\mathcal{C}(C_2^2,q_-)^{S_3}$ is not equivalent to $\mathrm{Rep}(S_4)$ as a fusion category as all braidings on $\mathrm{Rep}(S_4)$ are symmetric by \cite[Theorem 3.2]{MR3943750} (see Example \ref{braidfin}).

Otherwise, $\mathcal{O}(\mathrm{Rep}(S_3))$ acts trivially on $\{\rho_1,\rho_2\}:=\mathcal{O}(\mathcal{C})\setminus\mathcal{O}(\mathrm{Rep}(S_3))$, hence $[A\otimes \rho_j,A\otimes \rho_j]=[\rho_j,A\otimes \rho_j]=6$ \cite[Lemma 7.8.12]{tcat}.  Each of the free $A$-modules $A\otimes\rho_j$ thus decomposes into $3$ or $6$ distinct isomorphism classes of simple objects.  If both $A\otimes\rho_j$ for $j=1,2$ decompose into three, then $\mathcal{C}_A$ is rank $7$ with two pairs of two isomorphism classes of simple objects of the same dimension, i.e.\ those simple $X$ with $F(X)=\rho_j$, and one simple object each of twice these dimensions.  By inspection no such modular data exists \cite[Appendix E.6]{ng2023classification}.  If exactly one of $A\otimes\rho_j$ for $j=1,2$ decomposes into three nonisomorphic simple objects, then $\mathcal{C}_A$ is rank $10$ with six distinct isomorphism classes of simple objects of the same dimenson.  By inspection any such modular data is pointed \cite[Appendix E.6]{ng2023classification}; since the automorphism group of $C_{10}$ is cyclic of order $4$, there are no $S_3$-actions by braided autoequivalences with an orbit of size $6$, thus there are no examples in this case.  Lastly, if both $A\otimes\rho_j$ decompose into $6$ nonisomorphic simple objects, then $\dim(\mathcal{C}_A)=1+6d_1^2+6d_2^2$ for some algebraic integers $d_1,d_2$, thus $\dim(\mathcal{C})-d_1^2-d_2^2=1/6$ is an algebraic integer.  Therefore this case produces no examples either.

\paragraph{Rank $2$ symmetric center}

If there exists a Tannakian subcategory of rank $4$, it would necessarily contain the symmetric center by the balancing equation, hence any maximal Tannakian subcategory of $\mathcal{C}$ is rank less than or equal to $3$.  If there is a maximal Tannakian subcategory of rank $3$, its centralizer is a maximal fusion subcategory and the preceding arguments show no such category exists.  Moreover, $C_\mathcal{C}(\mathcal{C})$ is either a maximal Tannakian subcategory, or $C_\mathcal{C}(\mathcal{C})\simeq\mathrm{sVec}$.  In the latter case, the nontrivial invertible $\delta$ with $\theta_\delta=-1$ acts fixed-point-freely on $\mathcal{O}(\mathcal{C})$, thus $\mathrm{rank}(\mathcal{C})$ must be even, violating our assumptions.  Otherwise $C_\mathcal{C}(\mathcal{C})\simeq\mathrm{Rep}(C_2)$ and $\mathcal{C}_A$, where $A$ is the regular algebra of $\mathrm{Rep}(C_2)$, is a modular fusion category of rank $4$ or $7$, depending on whether the action of the nontrivial invertible on the $3$ isomorphism classes of simple objects not contained in the Tannakian subcategory has one fixed-point, or three, respectively.

\par Modular data of rank $7$ with three pairs of isomorphism classes of simple objects of the same dimensions are pointed \cite[Appendix E.6]{ng2023classification} which implies the dimensions of simple objects of $\mathcal{C}$ are $1,1,2,2,2$.  There are two braided inequivalent pointed modular fusion categories of rank $7$ each with a unique fixed-point free $C_2$-action by braided autoequivalences, giving two braided equivalence classes in this case.  As there is a unique pointed fusion category of rank $7$ with a braiding up to equivalence, both of these $C_2$-equivariantizations are equivalent to $\mathrm{Rep}(D_7)$ as fusion categories so we denote them $\mathrm{Rep}(D_7)^\psi$ where $\psi$ is a primitive $7$th root of unity.  Otherwise $\mathcal{C}_A$ is modular of rank $4$ and has a pair of simple objects of the same dimension and twist.  Therefore $\mathcal{C}_A$ is pointed on the group $C_2^2$ or $\mathcal{C}_A\simeq\mathcal{C}(A_1,5,q)^{\boxtimes2}$ is a braided equivalence where $q^2$ is a primitive fifth root of unity.  In the latter case, $\mathcal{C}$ would have Frobenius-Perron dimensions of simple objects $1,1,2[3]_5,[3]_5^2,[3]_5^2$.  It is elementary to deduce the fusion rules of $\mathcal{C}$ are those of $\mathcal{C}(A_1,10,q)_\mathrm{ad}$ where $q^2$ is any primitive $10$th root of unity.  There are four braided equivalence classes of such categories.  Otherwise we may then conclude $\mathcal{C}$ is a Tambara-Yamagami braided fusion category.  There are $12$ braided equivalence classes of such categories, all which can be realized as the rational subcategories of various products of Ising braided fusion categories \cite[Figure 5]{argentina}.

\paragraph{Rank $1$ symmetric center}

Up to Galois conjugacy, there are only five possible modular data \cite[Appendix E.4]{ng2023classification} of rank $5$.  One set of the five modular data is pointed, producing two braided equivalence classes (see Example \ref{ex:braidgroup}).  Two of the sets of five modular data are those of minimal modular extensions of $\mathrm{Rep}(S_3)^\omega$ from Figure \ref{fig:catranklessthan4}; since $C_\mathcal{C}(\mathrm{Rep}(S_3)^\omega)\simeq\mathrm{Rep}(C_2)$ is a braided equivalence, there are at most $|H^3(C_2,\mathbb{C}^\times)|=2$ different braided equivalence classes of such categories for each $\mathrm{Rep}(S_3)^\omega$ \cite[Theorem 4.22]{MR3613518}.  Denote the elements of $H^3(C_2,\mathbb{C}^\times)$ by $\gamma$ for use in Figure \ref{fig:catrank5}.  The embeddings of $\mathrm{Rep}(S_3)^\omega$ into these extensions are unique, thus there are two braided inequivalent $C_2$-extensions of each $\mathrm{Rep}(S_3)^\omega$. The eight modular data come from the fact that there are two distinct spherical structures on each of the four aforementioned categories.  One set of the five modular data includes the $10$ modular data of the transitive modular fusion categories $\mathcal{C}(A_1,11,q)_\mathrm{ad}$, classified by \cite{MR4389082}.  The last remaining set of the five modular data includes the six Galois conjugates of $\mathcal{C}(A_2,7,q)_\mathrm{ad}$.  Any braided fusion category with these modular data is braided equivalent to one of these six by \cite[Theorem 5.4]{MR4635615}.


\subsection{Premodular categories of rank $6$}

Here we organize our classification slightly differently than the preceding sections for ease of reading.  Our argument is based on the rank of any Tannakian fusion subcategory of maximal rank in decreasing order, and subdivide each case into arguments based on how the symmetric center relates to such a Tannakian fusion subcategory.


\subsubsection{Tannakian subcategory has maximal rank $5$ or $6$}\label{subsec5or6}

\begin{figure}[H]
\centering
\begin{align*}
\begin{array}{|c|c|c|c|c|}
\hline \mathcal{C} & \mathrm{FPdim}(\mathcal{C}) & \mathrm{FPdims} & C_\mathcal{C}(\mathcal{C}) & \# \\\hline\hline
\mathrm{Rep}(C_6) & 6 & 1,1,1,1,1,1 & \mathrm{Rep}(C_6) & 1 \\
\mathrm{Rep}(D_6) & 12 & 1,1,1,1,2,2 & \mathrm{Rep}(D_6) & 1 \\
\mathrm{Rep}(\mathrm{Dic}_3) & 12 & 1,1,1,1,2,2 & \mathrm{Rep}(\mathrm{Dic}_3) & 1\\
\mathrm{Rep}(D_9) & 18 & 1,1,2,2,2,2 & \mathrm{Rep}(D_9) & 1\\
\mathrm{Rep}(C_3\rtimes S_3) & 18 & 1,1,2,2,2,2 & \mathrm{Rep}(C_3\rtimes S_3) & 1\\
\mathrm{Rep}(C_3^2\rtimes C_4) & 36 & 1,1,1,1,4,4 & \mathrm{Rep}(C_3^2\rtimes C_4) & 1\\
\mathrm{Rep}(\mathrm{PSU}(3,2)) & 72 & 1,1,1,1,2,8 & \mathrm{Rep}(\mathrm{PSU}(3,2)) & 1\\
\mathrm{Rep}(\mathrm{GL}(3,2)) & 168 & 1,3,3,6,7,8 & \mathrm{Rep}(\mathrm{GL}(3,2)) & 1\\
\hline
\end{array}
\end{align*}
    \caption{The $8$ braided equivalence classes of premodular fusion categories of rank $6$ with Tannakian subcategory $\mathcal{D}$ of maximal rank $6$; there are none with $\mathrm{rank}(\mathcal{D})=5$.}%
    \label{fig:catrank56}%
\end{figure}

\par Tannakian categories of rank $6$ are described in Figure \ref{fig:symm} above.  Now assume $\mathrm{Rep}(G)\simeq\mathcal{D}\subset\mathcal{C}$ is a maximal Tannakian subcategory of rank $5$.  If $\mathcal{C}$ is symmetrically braided, then $\mathcal{C}$ would be one of the three categories described in Figure \ref{fig:symm} with $\nu\neq e$: $\mathrm{Rep}(C_6,\nu)$, $\mathrm{Rep}(D_6,\nu)$ or $\mathrm{Rep}(\mathrm{Dic}_3,\nu)$.  But none have Tannakian subcategories of rank $5$.  Otherwise, $\mathcal{C}=\mathcal{C}(\mathcal{D},\kappa)$ is a non-symmetrically braided near-integral fusion category for some $\kappa\in\mathbb{Z}_{\geq0}$.  Evidently $\mathcal{C}$ is not a rank $2$ category, so Proposition \ref{prop:fib} implies $\dim(\mathcal{C})\in\mathbb{Z}$. Proposition \ref{prop:tan} then states the only options for $\kappa\in\mathbb{Z}$ are $0$, $\sqrt{N/2}$ or $\sqrt{4N/3}$.  As $N$ is restricted to the orders in the second column of Figure \ref{fig:symm}, the latter roots are integers in only one case: $N=8$, $\kappa=2$, $d_+=4$, and $\dim(\mathcal{C})=24$.  The formal codegrees of such a fusion category are $24,12,8,8,8,4$ by Corollary \ref{cor:form} which implies the simple summands of the induction of the unit $I(e)$ have dimensions $1,2,3,3,3,6$ \cite[Theorem 2.13]{ost15}.  But $[\rho,F(I(e))]=2$, and only the simple object of dimension $6$ can forget to a single copy of $\rho$.  So no such fusion category exists.  Lastly, if $\kappa=0$ then $N$ must be a perfect square as $\mathcal{C}$ is integral, which never occurs by observing $|G|$ is not a perfect square for the groups in the second column of Figure \ref{fig:symm}.


\subsubsection{Tannakian subcategory has maximal rank $4$}

\par Let $\mathrm{Rep}(G)\simeq\mathcal{D}\subset\mathcal{C}$ be any Tannakian fusion subcategory of rank $4$ and let $\rho_1,\rho_2\in\mathcal{O}(\mathcal{C})$ be the isomorphism classes of simple objects not contained in $\mathcal{D}$.  We will first prove that with four exceptions unrelated to our classification of rank $6$ premodular fusion categories, if $\mathrm{rank}(\mathcal{D})+2=\mathrm{rank}(\mathcal{C})$ for any Tannakian subcategory $\mathcal{D}\subset\mathcal{C}$ of a premodular fusion category $\mathcal{C}$, then $\mathcal{C}$ is integral or the symmetric center of $\mathcal{C}$ acts trivially on $\rho_1,\rho_2$ (Lemma \ref{lem:4}).  This result allows a systematic study of premodular fusion categories in these two cases.  The resulting classification is given in Figure \ref{fig:rank456}.

\begin{figure}[H]
\centering
\begin{align*}
\begin{array}{|c|c|c|c|c|}
\hline \mathcal{C} & \mathrm{FPdim}(\mathcal{C}) & \mathrm{FPdims} & C_\mathcal{C}(\mathcal{C}) & \# \\\hline\hline
\mathrm{Rep}(C_2)\boxtimes\mathrm{Rep}(S_3)^\omega & 12 & 1,1,1,1,2,2 & \mathrm{Rep}(C_2^2) & 2\\
\mathrm{Rep}(\mathrm{Dic}_3)^\omega & 12 & 1,1,1,1,2,2 & \mathrm{Rep}(C_4) & 2\\
\mathrm{Rep}(C_3^2\rtimes C_4)^\omega & 36 & 1,1,1,1,4,4 & \mathrm{Rep}(C_4) & 1 \\
\hline
\end{array}
\end{align*}
    \caption{The $5$ braided equivalence classes of premodular fusion categories of rank $6$ with Tannakian subcategory of maximal rank $4$}%
    \label{fig:rank456}%
\end{figure}

\paragraph{Tannakian subcategories $\mathcal{D}\subset\mathcal{C}$ with $\mathrm{rank}(\mathcal{D})+2=\mathrm{rank}(\mathcal{C})$}

\begin{lemma}\label{lem:quadr}
Let $\mathcal{C}$ be a fusion category and $\mathcal{D}\subset\mathcal{C}$ be an integral fusion subcategory of $\mathcal{C}$ with $\mathrm{rank}(\mathcal{D})+2=\mathrm{rank}(\mathcal{C})$.  Denote the elements of $\mathcal{O}(\mathcal{C})\setminus\mathcal{O}(\mathcal{D})$ by $\rho_1,\rho_2$.  If there exists $X\in\mathcal{O}(\mathcal{D})$ such that $N_{X,\rho_1}^{\rho_2}$ or $N_{X,\rho_2}^{\rho_1}$ is nonzero, then $\dim(\rho_1)$ and $\dim(\rho_2)$ are quadratic integers (possibly rational).
\end{lemma}

\begin{proof}
Since $\mathcal{D}$ is a fusion subcategory, for any $X,Y\in\mathcal{O}(\mathcal{D})$, $0=N_{Y^\ast,X}^{\rho_j^\ast}=N_{X,\rho_j}^Y$.  So the only potentially nonzero fusion coefficients for $X\otimes\rho_1$ are $N_{X,\rho_1}^{\rho_1}$ and $N_{X,\rho_1}^{\rho_2}$.  Therefore
\begin{align}
&&\dim(X)\dim(\rho_1)&=N_{X,\rho_1}^{\rho_1}\dim(\rho_1)+N_{X,\rho_1}^{\rho_2}\dim(\rho_2) \\
\Rightarrow&&q:=\dfrac{\dim(X)-N_{X,\rho_1}^{\rho_1}}{N_{X,\rho_1}^{\rho_2}}&=\dfrac{\dim(\rho_2)}{\dim(\rho_1)}\in\mathbb{Q},
\end{align}
if $N_{X,\rho_1}^{\rho_2}\neq0$, without loss of generality.  Therefore by measuring $\dim(\rho_1\otimes\rho_1)$,
\begin{equation}
\dim(\rho_1)^2=r+s\dim(\rho_1)
\end{equation}
where $r\in\mathbb{Z}$ and $s=N_{\rho_1,\rho_1}^{\rho_1}+N_{\rho_1,\rho_1}^{\rho_2}q\in\mathbb{Q}$.  The argument is symmetric in $\rho_1$ and $\rho_2$. Moreover $\dim(\rho_1)$ and $\dim(\rho_2)$ are both quadratic integers.
\end{proof}

\begin{lemma}\label{lem:threee}
Let $\mathcal{C}$ be a premodular fusion category and $\mathcal{D}\subset\mathcal{C}$ be a Tannakian subcategory of $\mathcal{C}$ with $\mathrm{rank}(\mathcal{D})+2=\mathrm{rank}(\mathcal{C})$.  Then at least one of the following is true:
\begin{enumerate}
\item $\mathcal{C}$ is symmetrically braided,
\item $\mathcal{D}\subsetneq C_\mathcal{C}(\mathcal{C})$ and $C_\mathcal{C}(\mathcal{C})$ is a maximal fusion subcategory of $\mathcal{C}$, or
\item $C_\mathcal{C}(\mathcal{C})\subset\mathcal{D}$.
\end{enumerate}
\end{lemma}

\begin{proof}
We have $\mathcal{D}\subset\mathcal{C}$ which implies $C_\mathcal{C}(\mathcal{C})\subset C_\mathcal{C}(\mathcal{D})$.  Note that $\mathcal{D}\subset C_\mathcal{C}(\mathcal{D})$ as well since $\mathcal{D}$ is Tannakian, hence $\mathrm{rank}(C_\mathcal{C}(\mathcal{D}))\in\{\mathrm{rank}(\mathcal{D}),\mathrm{rank}(\mathcal{D})+1,\mathrm{rank}(\mathcal{C})\}$.  In the case $\mathrm{rank}(C_\mathcal{C}(\mathcal{D}))=\mathrm{rank}(\mathcal{C})$, then $\mathcal{D}\subset C_\mathcal{C}(\mathcal{C})$.  Therefore either $\mathcal{C}$ is symmetrically braided, $C_\mathcal{C}(\mathcal{C})$ is a maximal fusion subcategory, or $\mathcal{D}=C_\mathcal{C}(\mathcal{C})$ and we are done.   If $\mathrm{rank}(C_\mathcal{C}(\mathcal{D}))=\mathrm{rank}(\mathcal{D})+1$, then $C_\mathcal{C}(\mathcal{D})$ is the unique maximal fusion subcategory which would imply $\mathcal{D}=C_\mathcal{C}(\mathcal{D})$ by Propositions \ref{prop:tan} and \ref{prop:stan}, which cannot occur by our assumption of rank.  Otherwise $C_\mathcal{C}(\mathcal{D})=\mathcal{D}$.  Therefore $C_\mathcal{C}(\mathcal{C})\subset\mathcal{D}$.
\end{proof}

\begin{lemma}\label{lem:4}
Let $\mathcal{C}$ be a premodular fusion category and $\mathcal{D}\subset\mathcal{C}$ be a Tannakian subcategory of $\mathcal{C}$ with regular algebra $A$ and $\mathrm{rank}(\mathcal{D})+2=\mathrm{rank}(\mathcal{C})$.  Denote the elements of $\mathcal{O}(\mathcal{C})\setminus\mathcal{O}(\mathcal{D})$ by $\rho_1,\rho_2$.  If there exists $X\in\mathcal{O}(C_\mathcal{C}(\mathcal{C}))$ such that $N_{X,\rho_1}^{\rho_2}$ or $N_{X,\rho_2}^{\rho_1}$ is nonzero, then either
\begin{enumerate}
\item $\mathcal{C}_A$ is integral of Frobenius-Perron dimension $1$, $2$, $3$, or $4$, hence $\mathcal{C}$ is integral, or 
\item $\mathcal{C}\simeq\mathrm{Rep}(C_2)\boxtimes\mathcal{C}(A_1,5,q)_\mathrm{ad}$ where $q^2$ is a primitive $5$th root of unity.
\end{enumerate}
\end{lemma}

\begin{proof}
The first two category types in Lemma \ref{lem:threee} satisfy $\mathrm{FPdim}(\mathcal{C}_A)\in\{1,2,3,4\}$ by Propositions \ref{prop:tan} and \ref{prop:stan}, so we assume through the remainder of the proof that $C_\mathcal{C}(\mathcal{C})\subset\mathcal{D}$ and $\mathcal{C}_A$ is a modular fusion category. 

\par Let $X\in\mathcal{O}(C_\mathcal{C}(\mathcal{C}))$ satisfying the hypothesis of our statement.  By the balancing equation \cite[Proposition 8.13.8]{tcat}, since $X$ and $\rho_1$ centralize one another,
\begin{align}
&&\dim(X)\dim(\rho_1)&=S_{X,\rho_1}=N_{X,\rho_1}^{\rho_1}\dim(\rho_1)+\dfrac{\theta_{\rho_2}}{\theta_{\rho_1}}N_{X,\rho_1}^{\rho_2}\dim(\rho_2) \\
\Rightarrow&&\dim(X)-N_{X,\rho_1}^{\rho_1}&=\dfrac{\theta_{\rho_2}}{\theta_{\rho_1}}N_{X,\rho_1}^{\rho_2}\dfrac{\dim(\rho_2)}{\dim(\rho_1)}.
\end{align}
In particular, since $N_{X,\rho_1}^{\rho_2}\neq0$, then $\theta:=\theta_{\rho_1}=\theta_{\rho_2}$ and $\dim(\rho_1)/\dim(\rho_2)\in\mathbb{Q}$.  This means that the modular fusion category $\mathcal{C}_A$ consists of the tensor unit, and all other simple objects have the same twist.  We then compute the product of the Gauss sums of $\mathcal{C}_A$ \cite[Proposition 8.15.4]{tcat},
\begin{align}
1+M&=\dim(\mathcal{C}_A)=\rho_1(\mathcal{C}_A)\rho_{-1}(\mathcal{C}_A) \\
&=(1+\theta M)(1+\theta^{-1}M) \\
&=1+M(\theta+\theta^{-1})+M^2
\end{align}
where $M$ is the sums of the squares of the dimensions of the nontrivial simple objects of $\mathcal{C}_A$. Therefore $M=1-(\theta+\theta^{-1})$, or $\dim(\mathcal{C}_A)=2(1-\cos(2\pi/n))$ for some $n\in\mathbb{Z}_{\geq1}$.  We must have $\dim(\mathcal{C}_A)$ is a rational integer or quadratic integer by Lemma \ref{lem:quadr}, and the only $n$ such that $\dim(\mathcal{C}_A)$ is a nonzero integer or quadratic integer are $n\in\{2,3,4,5,6,8,10,12\}$.  The maximal of the Galois conjugates of $\dim(\mathcal{C}_A)\not\in\mathbb{Z}$ are $(1/2)(5+\sqrt{5})$, $2+\sqrt{2}$, $(1/2)(3+\sqrt{5})$ and $2+\sqrt{3}$.  The latter three have a Galois conjugate less than $1$, so these are not possible \cite[Theorem 2.3]{ENO}.  Modular fusion categories of dimension $(1/2)(5\pm\sqrt{5})$ are classified and must be $\mathcal{C}(A_1,5,q)_\mathrm{ad}$ where $q^2$ is a primitive fifth root of unity \cite[Example 5.1.2(iv)]{MR3943751}.  There are no nontrivial braided autoequivalences of $\mathcal{C}(A_1,5,q)_\mathrm{ad}$, thus $\mathcal{C}\simeq\mathcal{C}(A_1,5,q)_\mathrm{ad}\boxtimes\mathcal{D}$ which would imply $\mathrm{rank}(\mathcal{C})=2\cdot\mathrm{rank}(D)$, thus $\mathrm{rank}(\mathcal{D})=2$.  Otherwise $\dim(\mathcal{C}_A)\in\{1,2,3,4\}$.  All such modular fusion categories are integral except Ising categories which do not satisfy our hypotheses.  Therefore $\mathcal{C}_A$ is integral and moreover $\mathcal{C}$ is integral as well.
\end{proof}

\begin{example}
The smallest example not satisfying the fusion rule hypotheses of the above results are the modular rank 3 fusion categories.  In particular, $\mathcal{C}(A_1,7,q)_\mathrm{ad}$ where $q^2$ is a primitive $7$th root of unity do not satisfy the hypotheses of any of the above three lemmas; indeed the dimensions of $\rho_1$ and $\rho_2$ are cubic integers.
\end{example}


\paragraph{The rank 6 case with transitive $\mathcal{O}(C_\mathcal{C}(\mathcal{C}))$-action} Here we use Lemma \ref{lem:4} which implies $\mathcal{C}_A$ is integral of dimension $2$, $3$, or $4$ where $1$ has been eliminated since $\mathcal{C}$ cannot be super Tannakian by Figure \ref{fig:symm}.  If $\mathcal{D}$ is pointed, then $\dim(\mathcal{C})\in\{8,12,16\}$; there exist possible integer dimensions for $\rho_1,\rho_2$ only if $\dim(\mathcal{C})=12$ and $\dim(\rho_1)=\dim(\rho_2)=2$.  The following was pointed out to the authors by \mbox{G.\ Vercleyen} and is proven computationally by brute force in \cite{vercleyen2024lowrank}.

\begin{lemma}
There are exactly $4$ fusion categories with isomorphism classes of simple objects of dimensions $1,1,1,1,2,2$ up to equivalence which admit braidings.  Each possesses $6$ inequivalent braidings.
\end{lemma}

\begin{proof}
It is elementary to verify the only possible fusion rules with these dimensions are those of $\mathrm{Rep}(C_2\times S_3)$ or $\mathrm{Rep}(\mathrm{Dic}_3)$.  No such category is modular as there is no compatible modular data \cite[Appendix E.5]{ng2023classification} and is not supermodular either by \cite{MR4109138} (or see Section \ref{sec6mod}). hence the symmetric center contains a Tannakian subcategory braided equivalent to $\mathrm{Rep}(C_2)$ with regular algebra $A$ which fixes the simple objects of dimension $2$.  Hence the category $\mathcal{C}_A$ is a pointed braided fusion category with $\dim(\mathcal{C}_A)=6$, i.e.\ such a braided fusion category factors into a factor of dimension $2$ and a factor of dimension $3$.  The result then follows by the discussion in Example \ref{ex:12}.
\end{proof}

As shown in Example \ref{ex:12}, only four of these $24$ categories with inequivalent braidings have a symmetric center of rank $4$.  As fusion categories, two are equivalent to $\mathrm{Rep}(C_2\times S_3)$ and two are equivalent to $\mathrm{Rep}(\mathrm{Dic}_3)$, twisted by primitive third roots of unity $\omega$.

\par Recall that as a corollary to \cite[Proposition 8.14.6]{tcat}, we must have $\mathrm{FPdim}(X)$ divides $\mathrm{FPdim}(\mathcal{C})$ for all simple objects $X$ is a braided fusion category $\mathcal{C}$ due to the canonical embedding $\mathcal{C}\hookrightarrow\mathcal{Z}(\mathcal{C})$.  So lastly if $\mathcal{D}\simeq\mathrm{Rep}(D_5)$ is a braided equivalence, then $\dim(\mathcal{C})\in\{20,30,40\}$; there are no possible integer dimensions of $\rho_1,\rho_2$ in this case which divide $\dim(\mathcal{C})$.  And if $\mathcal{D}\simeq\mathrm{Rep}(A_4)$ is a braided equivalence, then $\dim(\mathcal{C})\in\{24,36,48\}$; there are no possible integer dimensions of $\rho_1,\rho_2$ in this case either.


\paragraph{The rank 6 case with trivial $\mathcal{O}(C_\mathcal{C}(\mathcal{C}))$-action}

Assume $\mathcal{C}$ is a rank 6 premodular fusion category with $\mathcal{D}\subset\mathcal{C}$ a Tannakian fusion subcategory of maximal rank $4$ such that $\mathcal{O}(C_\mathcal{C}(\mathcal{C}))$ acts trivially on $\rho_1,\rho_2$.  We will consider the three cases of Lemma \ref{lem:threee} separately.  If $\mathcal{C}$ is symmetrically braided, then $\mathcal{C}$ is super Tannakian and would be described in Figure \ref{fig:symm}; no such categories exist.  If $C_\mathcal{C}(\mathcal{C})$ is a maximal fusion subcategory, then $\mathrm{FPdim}(\mathcal{C})\in\mathbb{Z}$ by Proposition \ref{prop:fib} since $\mathcal{C}$ is not a rank $2$ fusion category and $C_\mathcal{C}(\mathcal{C})$ is super Tannakian by assumption.  Then Proposition \ref{prop:stan} implies $\mathcal{C}=\mathcal{C}(C_\mathcal{C}(\mathcal{C}),0)$ is a near-integral braided fusion category and the only possibilities for $C_\mathcal{C}(\mathcal{C})$ by Figure \ref{fig:symm} are $\mathrm{Rep}(D_4,\nu)$ and $\mathrm{Rep}(Q_8,\nu)$ for nontrivial $\nu$.  Moreover $\mathrm{FPdim}(\mathcal{C})=16$ and the isomorphism classes of simple objects have dimensions $1,1,1,1,2,2\sqrt{2}$.  Let $B\subset\mathcal{D}$ be the regular algebra of any $\mathrm{Rep}(C_2)$ fusion subcategory.  Then the de-equivariantization $\mathcal{C}_B$ has $\mathrm{FPdim}(\mathcal{C}_B)=8$ and $\mathcal{C}_B$ is supermodular of rank $6$ as $\mathrm{Rep}(C_2)$ is a Tannakian fusion subcategory of maximal dimension.  From the classification of such categories \cite{MR4109138}, $\mathcal{C}_B\simeq\mathrm{sVec}\boxtimes\mathcal{I}_q$ is a braided equivalence for a primitive $16$th root of unity $q$.  But the simple objects of $\mathcal{C}_B$ of dimension $\sqrt{2}$ have distinct twists hence there is no braided autoequivalence mapping one to other and moreover this case produces no examples.  It remains to consider the case $C_\mathcal{C}(\mathcal{C})\subset\mathcal{D}$ by Lemma \ref{lem:threee} and thus $\mathcal{C}_A$ is a modular fusion category as $\mathcal{D}$ must be of maximal dimension as well as rank in this case.  The only options are $\mathcal{D}$ is braided equivalent to $\mathrm{Rep}(G)$ where $G\cong C_4,C_2^2,D_5,A_4$ with regular algebra $A$ by Figure \ref{fig:symm}.

\par Assume first that $\mathcal{D}$ is pointed.  If $\mathcal{O}(\mathcal{D})$ acts trivially on $\rho_1,\rho_2$, then $\mathcal{C}_A$ is rank $9$ and has two sets of four simple objects which are all the same dimension and twist.  An exhaustive search of possible modular data \cite[Appendix E.5]{ng2023classification} shows that such a modular fusion category is pointed.  By the classification of metric groups of order $9$, $\mathcal{C}_A\simeq\mathcal{C}(C_3,q)^{\boxtimes2}$ where $q$ is either symmetric nondegenerate bilinear form on $C_3$.  Hence $\dim(\rho_1)=\dim(\rho_2)=4$, $\theta_{\rho_1}=\theta_{\rho_2}^{-1}=\omega$ for a primitive 3rd root of unity $\omega$, and $\dim(\mathcal{C})=36$.  But note that there is no $C_2^2$-action on $C_3^2$ which is fixed-point free, so $\mathcal{O}(\mathcal{D})$ must be cyclic.  The unique action of $C_4$ on $C_3^2$ by braided autoequivalences produces the fusion category $\mathrm{Rep}(C_3^2\rtimes C_4)$ which has two inequivalent braidings by \cite[Theorem 3.2]{MR3943750} (see Example \ref{braidfin}): the symmetric one, and the relevant nonsymmetric braiding twisted by either primitive third root of unity $\omega$.  If $\mathcal{O}(\mathcal{D})$ acts transitively on $\rho_1,\rho_2$, then $\mathcal{C}_A$ is rank $5$ with two pairs of two simple objects of the same dimension which by Figure \ref{fig:catrank5b} is pointed, thus $\dim(\mathcal{C})=20$ and $\mathcal{C}$ is integral.  But $\dim(\rho_1)=\sqrt{(\dim(\mathcal{C})-4)/2}\not\in\mathbb{Z}$, so this case produces no examples.

\par Now assume $\mathcal{D}\simeq\mathrm{Rep}(A_4)$, hence $C_\mathcal{C}(\mathcal{C})\simeq\mathrm{Rep}(G)$ where $G\cong C_1,C_3,A_4$.  By inspection of rank $6$ modular data \cite[Appendix E.5]{ng2023classification}, $G$ is not trivial.  If $G\cong C_3,A_4$, let $B$ be the regular algebra of $\mathrm{Rep}(C_3)\subset C_\mathcal{C}(\mathcal{C})$.  The de-equivariantization $\mathcal{C}_B$ is a rank $10$ braided fusion category with three nontrivial invertible objects of order 2 coming from the splitting of the simple object of $\mathrm{Rep}(A_4)$ of dimension $3$, and six simple objects appearing in sets of three: $\rho_i^j$ for $i\in\{1,2\}$ and $j\in\{1,2,3\}$ coming from the splitting of the simple objects $\rho_i$.  If all $\rho_i^j$ are the same dimension $d$, then $d\in\mathbb{Z}$ by \cite[Lemma 2.6]{MR4655273}.  Moreover $d$ is a root of $x^2-ax-b$ where $b$ divides $4$, hence $d\in\{1,2,4\}$.  As a consequence, the dimensions of the simple objects of $\mathcal{C}$ are $1,1,1,2,3,3$, $1,1,1,2,6,6$ or $1,1,1,2,12,12$.  None of these global dimensions is divisible by $\dim(\mathrm{Rep}(C_3))=3$, so we conclude $\dim(\rho_1^j)\neq\dim(\rho_2^k)$ for any $j,k\in\{1,2,3\}$.  There is no such rank 10 modular data \cite[Appendix E.9]{ng2023classification} so this case produces no examples.

\par Lastly assume $\mathcal{D}\simeq\mathrm{Rep}(D_5)$, hence $C_\mathcal{C}(\mathcal{C})\simeq\mathrm{Rep}(G)$ where $G\cong C_1,C_2,D_5$.  By inspection of possible rank $6$ modular data \cite[Appendix E.5]{ng2023classification}, $G$ is not trivial.  Let $B\subset\mathrm{Rep}(C_2)\subset C_\mathcal{C}(\mathcal{C})$ be the regular algebra so that $\mathcal{C}_A$ is a braided fusion category of rank $9$ with $4$ nontrivial invertible objects of order $5$ and four simple objects appearing in sets of two: $\rho_i^j$ for $i\in\{1,2\}$ and $j\in\{1,2\}$ coming from the splitting of the simple objects $\rho_i$.  First note that if $\mathcal{C}$ were not integral, then there exist Galois conjugates of $\mathbbm{1}$ with dimension distinct from $1$ by \cite[Theorem 4.1.6]{plavnik2021modular}, hence the number of such simple objects is a multiple of $5$ which cannot occur under these circumstances.  Thus $\mathcal{C}_B$ is integral.  Let $d_i:=\dim(\rho_i^j)$ so that $\dim(\mathcal{C}_B)=5+2d_1+2d_2$.  In particular, $\dim(\mathcal{C}_B)$ is odd and therefore $\mathcal{C}_B$ is pointed \cite{CZENKY2023}.  Moreover $\dim(\mathcal{C})=18$ with a fusion subcategory of dimension $\dim(\mathrm{Rep}(D_5))=10$ which cannot occur.


\subsubsection{Tannakian subcategory has maximal rank $3$}\label{sec:repc3}

\begin{figure}[H]
\centering
\begin{align*}
\begin{array}{|c|c|c|c|c|}
\hline \mathcal{C} & \mathrm{FPdim}(\mathcal{C}) & \mathrm{FPdims} & C_\mathcal{C}(\mathcal{C}) & \# \\\hline\hline
\mathrm{sVec}\boxtimes\mathrm{Rep}(C_3) & 6 & 1,1,1,1,1,1 & \mathrm{Rep}(C_6,\nu) & 1 \\
\mathcal{C}(C_2,q)\boxtimes\mathrm{Rep}(C_3) & 6 & 1,1,1,1,1,1 & \mathrm{Rep}(C_3) & 2\\
\mathcal{C}(A_1,5,q)_\mathrm{ad}\boxtimes\mathrm{Rep}(C_3) & \frac{15}{4}\csc^2(\frac{\pi}{5}) & 1,1,1,[3]_5,[3]_5,[3]_5 & \mathrm{Rep}(C_3) & 4\\
\mathcal{C}(C_2,q)\boxtimes\mathrm{Rep}(S_3) & 12 & 1,1,1,1,2,2 & \mathrm{Rep}(S_3) & 2\\
\mathrm{sVec}\boxtimes\mathrm{Rep}(S_3) & 12 & 1,1,1,1,2,2 & \mathrm{Rep}(S_3) & 1\\
\mathrm{Rep}(C_3\rtimes S_3)^\omega & 18 & 1,1,2,2,2,2& \mathrm{Rep}(C_2) & 3 \\
\mathrm{Rep}(D_9)^\beta & 18 & 1,1,2,2,2,2& \mathrm{Rep}(C_2) & 4 \\
\mathcal{C}(A_1,5,q)_\mathrm{ad}\boxtimes\mathrm{Rep}(S_3) & \frac{15}{2}\csc^2(\frac{\pi}{5}) & 1,1,[3]_5,[3]_5,2,2[3]_5 & \mathrm{Rep}(S_3) & 4\\
\hline
\end{array}
\end{align*}
    \caption{The $21$ braided equivalence classes of premodular fusion categories of rank $6$ with Tannakian subcategory $\mathcal{D}$ of maximal rank $3$}%
    \label{fig:rank34567}%
\end{figure}
\par There are only two rank 3 Tannakian categories; each case is described separately below.

\paragraph{The case of $\mathrm{Rep}(C_3)$} Assume first that $\mathrm{Rep}(C_3)\simeq\mathcal{D}\subset\mathcal{C}$ is a Tannakian subcategory of maximal rank and let $\rho_1,\rho_2,\rho_3\in\mathcal{O}(\mathcal{C})$ be the isomorphism classes of simple objects not lying in $\mathcal{O}(\mathcal{D})$.  There are two cases to consider based on the $\otimes$-action of $\mathcal{O}(\mathcal{D})$ on $\{\rho_1,\rho_2,\rho_3\}$.

\par If $\mathcal{O}(\mathcal{D})$ acts transitively on the $\{\rho_1,\rho_2,\rho_3\}$, then $\mathcal{C}$ is a braided generalized near-group fusion category.  In particular, $\mathrm{FPdim}(\rho_1)=\mathrm{FPdim}(\rho_2)=\mathrm{FPdim}(\rho_3)$ and since $\mathcal{O}(\mathcal{D})$ acts on the $\{\rho_1,\rho_2,\rho_3\}$ without fixed-points, measuring the Frobenius-Perron dimension of $\rho_1\otimes\rho_1^\ast$ implies that $\mathrm{FPdim}(\rho_1)$ is a root of $x^2-ax-1$ for some $a\in\mathbb{Z}_{\geq0}$.  By \cite{MR4658217}, if $\mathrm{FPdim}(\rho_1)\not\in\mathbb{Z}$, then $\mathcal{C}\simeq\mathcal{D}\boxtimes\mathcal{C}(A_1,5,q)_\mathrm{ad}$ with $q^2$ a primitive fifth root of unity.  Otherwise, $\mathrm{FPdim}(\rho_1)\in\mathbb{Z}$ implies $\mathrm{FPdim}(\rho_1)=1$ by factoring the above quadratic over $\mathbb{Z}$ and thus $\mathcal{C}$ is pointed.  Moreover $\mathcal{C}\simeq\mathcal{D}\boxtimes\mathcal{E}$ where $\mathcal{E}$ is a rank 2 non-Tannakian pointed braided fusion category, which is either modular or braided equivalent to $\mathrm{sVec}$.

\par Otherwise, $\rho_1,\rho_2,\rho_3$ are fixed by $\mathcal{O}(\mathcal{D})$ under the $\otimes$-action, thus $\mathcal{D}\subset C_\mathcal{C}(\mathcal{C})$ by \cite[Proposition 8.20.5]{tcat} and $\mathcal{C}$ is not pointed.  Since there are no super Tannakian categories of ranks 4--6 which are not pointed containing exactly three invertible simple objects by Figure \ref{fig:symm}, then $C_\mathcal{C}(\mathcal{C})=\mathcal{D}$ is Tannakian.  Moreover with $A\in\mathcal{D}$ the regular algebra of $\mathrm{Rep}(C_3)$, $\mathcal{C}_A$ is a modular fusion category of rank $10$ with simple object $\mathbbm{1}$, and the dimensions of the remaining simple objects occur in sets of $3$.  By inspection \cite[Appendix E.9]{ng2023classification} the only such categories satisfying the above conditions are pointed.  There is no fixed-point free automorphism of the cyclic group of order $10$, so any braided autoequivalence of order dividing $3$ acting on a pointed modular fusion category of rank $10$ must have at least $4$ orbits in addition to the tensor unit, which would produce at least $7$ simple objects in the equivariantization \cite[Remark 4.15.8]{tcat}.  Hence no such premodular category exists.


\paragraph{The case of $\mathrm{Rep}(S_3)$}  Assume now that $\mathrm{Rep}(S_3)\simeq\mathcal{D}\subset\mathcal{C}$ is a Tannakian subcategory of maximal rank and let $\rho_1,\rho_2,\rho_3\in\mathcal{O}(\mathcal{C})$ be the isomorphism classes of simple objects not lying in $\mathcal{O}(\mathcal{D})$.  We will show that $\mathcal{D}_\mathrm{pt}\simeq\mathrm{Rep}(C_2)\subset C_\mathcal{C}(\mathcal{C})$.  Consider the relative centralizer subcategory $C_\mathcal{C}(\mathcal{D}_\mathrm{pt})$, which contains $\mathcal{D}$ by definition.  Since the action of $\mathcal{O}(\mathcal{D}_\mathrm{pt})$ on $\{\rho_1,\rho_2,\rho_3\}$ is either trivial or has one fixed-point, then $\mathrm{rank}(C_\mathcal{C}(\mathcal{D}_\mathrm{pt}))$ is at least 4 by the balancing equation.  Premodular fusion categories of rank 4 were described above in Figure \ref{fig:catrank4}, and none contain a fusion subcategory equivalent to $\mathrm{Rep}(S_3)$.  If $\mathrm{rank}(C_\mathcal{C}(\mathcal{D}_\mathrm{pt})))=5$, then $\mathcal{C}$ is a braided near-integral fusion category with maximal fusion subcategory $\mathcal{E}:=C_\mathcal{C}(\mathcal{D}_\mathrm{pt})$.  We have shown if $C_\mathcal{E}(\mathcal{E})$ is Tannakian, then $\mathcal{E}$ is the symmetric center of $\mathcal{C}$.   There are only 8 options for the fusion rules of $\mathcal{E}$ in Figure \ref{fig:symm}.  There must exist a Tannakian subcategory of $\mathcal{E}$ of global dimension $\dim(\mathcal{E})$ or $(1/2)\dim(\mathcal{E})$, and none of the possible options from Figure \ref{fig:symm} is compatible with $\mathcal{D}$ being the maximal Tannakian subcategory of $\mathcal{C}$.  Our conclusion is that $\mathcal{D}_\mathrm{pt}\subset C_\mathcal{C}(\mathcal{C})$.  This implies with $A$ the regular algebra of $\mathcal{D}_\mathrm{pt}$, the de-equivariantization $\mathcal{C}_A$ is a braided fusion category.  We may now proceed to analyze the cases when $\mathcal{D}_\mathrm{pt}$ acts on $\{\rho_1,\rho_2,\rho_3\}$ trivially, or with a unique fixed-point. 

\par In the case that $\mathcal{D}_\mathrm{pt}$ acts on $\{\rho_1,\rho_2,\rho_3\}$ with a unique fixed-point, $\mathrm{rank}(\mathcal{C}_A)=6$ and $\mathcal{D}_A\simeq\mathrm{Rep}(C_3)\subset\mathcal{C}_A$ is a Tannakian subcategory.  Note that $\mathcal{C}_A$ is not Tannakian or else $\mathcal{C}$ itself is, and the maximal Tannakian subcategory of $\mathcal{C}_A$ cannot have rank 5 as we have shown above there are no such braided fusion categories in Section \ref{subsec5or6}.  In the case the maximal Tannakian subcategory of $\mathcal{C}_A$ is rank 3, then it is $\mathcal{D}_A$ and the above work shows that $\mathcal{C}_A\simeq\mathcal{E}\boxtimes\mathrm{Rep}(C_3)$ is a braided equivalence where $\mathcal{E}$ is braided equivalent to $\mathrm{sVec}$, $\mathcal{C}(C_2,q)$ with $q$ a nondegenerate quadratic form on $C_2$, or $\mathcal{C}(A_1,5,q)_\mathrm{ad}$ with $q^2$ a primitive fifth root of unity.  The $C_2$-action by braided autoequivalences on $\mathcal{E}$ is trivial in the latter case, hence $\mathcal{C}\simeq\mathcal{C}(A_1,5,q)_\mathrm{ad}\boxtimes\mathrm{Rep}(S_3)$ is a braided equivalence.  Otherwise $\dim(\mathcal{C})=12$ and these cases are described in Example \ref{ex:12}.  It remains to consider if the maximal Tannakian subcategory of $\mathcal{C}_A$ has rank 4.  As $\mathcal{D}_A\subset\mathcal{C}_A$, then this Tannakian subcategory is equivalent to $\mathrm{Rep}(A_4)$ by Figure \ref{fig:symm}.  Moreover the simple objects of $\mathcal{C}$ have dimensions $1,1,2,3,3,d$.  If the set of simple objects without the one of dimension $d$ form a fusion subcategory $\mathcal{G}\subset\mathcal{C}$, then $C_\mathcal{G}(\mathcal{G})=\mathcal{G}$ is Tannakian, contradicting the fact that $\mathrm{Rep}(S_3)$ is the maximal Tannakian subcategory.  Therefore $d<10$, $d\in\mathbb{Z}$, and $2$ divides $d$.  This leaves $d\in\{2,4,6,8\}$.  But only $d=6$ allows $6=\dim(\mathcal{D})$ to divide $\dim(\mathcal{C})$ \cite[Proposition 8.15]{ENO}.  Lastly, we point out that if the symmetric center of $\mathcal{C}$ is $\mathcal{D}_\mathrm{pt}$, then $\mathcal{C}_A$ is a modular fusion category with simple objects of dimension $1,1,1,3,3,3$, which does not exist since the adjoint subcategory would have global dimension $10$ and thus rank $2$, which is absurd.  Since $\mathcal{G}$ cannot be the symmetric center of $\mathcal{C}$ as explained above, this implies $\mathcal{D}=C_\mathcal{C}(\mathcal{C})$.  In this case $\mathcal{C}_A$ is a braided fusion category with dimensions $1,1,1,3,3,3$ and maximal Tannakian subcategory $\mathrm{Rep}(C_3)$ which we have shown does not exist in in the preceding subsection.

\par Otherwise $\mathcal{D}_\mathrm{pt}\simeq\mathrm{Rep}(C_2)$ acts trivially on $\{\rho_1,\rho_2,\rho_3\}$, but $C_\mathcal{C}(\mathcal{C})$ may be equal to $\mathcal{D}_\mathrm{pt}$ or strictly larger.  If $\mathcal{D}_\mathrm{pt}=C_\mathcal{C}(\mathcal{C})$ and $A$ is the regular algebra of $\mathcal{D}_\mathrm{pt}$, then $\mathcal{C}_A$ is a modular fusion category of rank 9 with at least $3$ invertible objects, and the other $6$ simple objects have dimensions which occur in pairs.  By inspection \cite[Appendix E.8]{ng2023classification}, the only such modular fusion categories are pointed.  Thus the dimensions of $\mathcal{C}$ are $1,1,2,2,2,2$.  These categories are described in Example \ref{ex:12222}.

\par Now it remains to consider when $C_\mathcal{C}(\mathcal{C})$ is strictly larger than $\mathcal{D}_\mathrm{pt}$.  One possibility is that $\mathcal{D}_\mathrm{pt}\subset C_\mathcal{C}(\mathcal{D})$ but $\mathcal{D}\not\subset C_\mathcal{C}(\mathcal{D})$ in which case $C_\mathcal{C}(\mathcal{D})$ has rank $4$ or $5$ and contains $\mathcal{D}$ as a fusion subcategory.  Following the reasoning of the first paragraph, there are no premodular categories of rank $4$ containing a fusion subcategory equivalent to $\mathrm{Rep}(S_3)$.  Thus $C_\mathcal{C}(\mathcal{D})$ is a maximal fusion subcategory.  Since $C_\mathcal{C}(\mathcal{D})$ being Tannakian would contradict the assumption $\mathcal{D}$ is maximal Tannakian, then $C_\mathcal{C}(\mathcal{D})$ is super Tannakian containing a Tannakian subcategory of dimension $6$.  Moreover $\dim(C_\mathcal{C}(\mathcal{D}))=12$, but there are no such symmetrically braided rank 5 categories.  Therefore $\mathcal{D}\subset C_\mathcal{C}(\mathcal{C})$.  And if $\mathcal{D}\subsetneq C_\mathcal{C}(\mathcal{C})$, $C_\mathcal{C}(\mathcal{C})$ is super Tannakian of dimension $12$, which we know does not exist.  Hence $\mathcal{D}=C_\mathcal{C}(\mathcal{C})\simeq\mathrm{Rep}(S_3)$ is a braided equivalence.

\par Consider the $C_2$-de-equivariantization by $\mathcal{D}_\mathrm{pt}$, i.e.\ $\mathcal{C}_A$ with $A$ the regular algebra of $\mathcal{D}_\mathrm{pt}$.  This gives a premodular category of rank $9$ with symmetric center $C_{\mathcal{C}_A}(\mathcal{C}_A)\simeq\mathrm{Rep}(C_3)$ and six simple objects split from the trivial action of $\mathcal{D}_\mathrm{pt}$ on $\{\rho_1,\rho_2,\rho_3\}$; call these corresponding simple objects $\rho_j^1,\rho_j^2$ for $j=1,2,3$.   Let $B\subset\mathrm{Rep}(C_3)$ be the regular algebra.  There are three possibilities for the action of $B$ on $\{\rho_j^1,\rho_j^2:j=1,2,3\}$.  Note that $(\mathcal{C}_A)_B$ is a modular fusion category.

\begin{figure}[H]
\centering
\begin{align*}
\begin{tikzpicture}[scale=0.75]
\node (a) at (0,0) {$\bullet$};
\node (b) at (1,0) {$\bullet$};
\node (c) at (2,0) {$\bullet$};
\node (d) at (0,1) {$\bullet$};
\node (e) at (1,1) {$\bullet$};
\node (f) at (2,1) {$\bullet$};
\node (g) at (0,2) {$\bullet$};
\node (h) at (1,2) {$\bullet$};
\node (i) at (2,2) {$\bullet$};
\draw[<->] (a) -- (b);
\draw[<->] (b) -- (c);
\draw[<->,out=45,in=135,looseness=1] (c) edge (a);
\draw[<->] (d) -- (e);
\draw[<->] (e) -- (f);
\draw[<->,out=45,in=135,looseness=1] (f) edge (d);
\draw[<->] (g) -- (h);
\draw[<->] (h) -- (i);
\draw[<->,out=45,in=135,looseness=1] (i) edge (g);
\node (a1) at (4,0) {$\bullet$};
\node (b1) at (5,0) {$\bullet$};
\node (c1) at (6,0) {$\bullet$};
\node (d1) at (4,1) {$\bullet$};
\node (e1) at (5,1) {$\bullet$};
\node (f1) at (6,1) {$\bullet$};
\node (g1) at (4,2) {$\bullet$};
\node (h1) at (5,2) {$\bullet$};
\node (i1) at (6,2) {$\bullet$};
\path[->,out=45,in=135,looseness=5] (g1) edge (g1);
\path[->,out=45,in=135,looseness=5] (h1) edge (h1);
\path[->,out=45,in=135,looseness=5] (i1) edge (i1);
\draw[<->] (a1) -- (b1);
\draw[<->] (b1) -- (c1);
\draw[<->,out=45,in=135,looseness=1] (c1) edge (a1);
\draw[<->] (d1) -- (e1);
\draw[<->] (e1) -- (f1);
\draw[<->,out=45,in=135,looseness=1] (f1) edge (d1);
\node (a2) at (8,0) {$\bullet$};
\node (b2) at (9,0) {$\bullet$};
\node (c2) at (10,0) {$\bullet$};
\node (d2) at (8,1) {$\bullet$};
\node (e2) at (9,1) {$\bullet$};
\node (f2) at (10,1) {$\bullet$};
\node (g2) at (8,2) {$\bullet$};
\node (h2) at (9,2) {$\bullet$};
\node (i2) at (10,2) {$\bullet$};
\path[->,out=45,in=135,looseness=5] (g2) edge (g2);
\path[->,out=45,in=135,looseness=5] (h2) edge (h2);
\path[->,out=45,in=135,looseness=5] (i2) edge (i2);
\draw[<->] (a2) -- (b2);
\draw[<->] (b2) -- (c2);
\draw[<->,out=45,in=135,looseness=1] (c2) edge (a2);
\path[->,out=45,in=135,looseness=5] (d2) edge (d2);
\path[->,out=45,in=135,looseness=5] (e2) edge (e2);
\path[->,out=45,in=135,looseness=5] (f2) edge (f2);
\end{tikzpicture}
\end{align*}
    \caption{An illustration of the three cases of the $C_3$-action on $\mathcal{O}(\mathcal{C}_A)$}%
    \label{fig:rank3action}%
\end{figure}
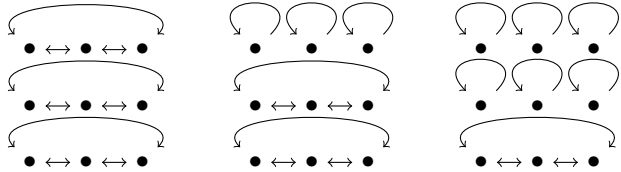

\par\noindent\textbf{$C_3$-action has no fixed-points}  If $B$ acts freely on $\{\rho_j^1,\rho_j^2:j=1,2,3\}$, then by the pigeon-hole principle, the dimensions of all $\rho_j^k$ for $j=1,2,3$ and $k=1,2$ are equal.  There are exactly two nontrivial simple objects of $(\mathcal{C}_A)_B$ which are the free $B$-modules $B\otimes\rho_j^k$ over all $j,k$.  There are only three fusion rules for rank $3$ modular fusion categories and only the pointed modular fusion categories of rank 3 have two nontrivial simple objects of the same dimension.  Hence $(\mathcal{C}_A)_B$ is pointed, $\mathcal{C}_A$ is pointed of rank $9$, thus this case is described in Example \ref{ex:12222}.

\par\noindent\textbf{$C_3$-action has three fixed-points} If $B$ acts with one orbit of size 3 and 3 fixed-points on $\{\rho_j^1,\rho_j^2:j=1,2,3\}$, then $(\mathcal{C}_A)_B$ is rank 11 with the trivial object of dimension $1$, four simple objects of dimension $d_1$ and six simple objects of dimension $d_2$ for some $d_1,d_2\in\mathbb{R}$ (possibly equal).  By inspection \cite[Appendix E(10)]{ng2023classification} there are no such modular data other than those corresponding to the pointed modular fusion categories.  This implies $\mathcal{C}_A$ has six invertible objects and three simple objects of dimension $3$.  But $\dim((\mathcal{C}_A)_\mathrm{pt})=6\nmid33=\dim(\mathcal{C}_A)$ so no such category exists. 

\par\noindent\textbf{$C_3$-action has six fixed-points} The remaining case is that $B$ acts trivially on $\{\rho_j^1,\rho_j^2:j=1,2,3\}$.  Equivalently, the simple objects of $\mathcal{D}$ must act trivially on $\rho_1,\rho_2,\rho_3$.  The modular fusion category $(\mathcal{C}_A)_B$ is rank $19$, but has (at most) three dimensions of simple objects distinct from $1$ and the full twists on the six simple objects of coinciding dimension are the same as $\rho_1,\rho_2,\rho_3$.  Note that if $(\mathcal{C}_A)_B$ is not integral, then at least $7$ simples objects $X\in(\mathcal{C}_A)_B$ have $\dim(X)=1$.  But their nontrivial Galois conjugates would be at least $7$ in number, thus there are only two distinct dimensions of simple objects by the pigeon-hole principle, each occurring with the same multiplicity.  As this cannot occur for a fusion category of rank $19$ there are no non-integral examples.  Thus we conclude $(\mathcal{C}_A)_B$ is integral, and completely anisotropic, since $\mathcal{D}$ was the maximal Tannakian fusion subcategory of $\mathcal{C}$.  Therefore, if $(\mathcal{C}_A)_B$ has any nontrivial invertible objects and is not pointed, the pointed subcategory is modular and $(\mathcal{C}_A)_B$ would factor since the pointed subcategory can only have ranks $7$ or $13$.  But $\mathrm{rank}((\mathcal{C}_A)_B)=19$ is prime so either $(\mathcal{C}_A)_B$ is pointed, or it has no nontrivial invertible objects.  In the pointed case, there is a unique pointed fusion category of rank $19$ which admits braidings, $\mathrm{Vec}_{C_{19}}$ with trivial associator.  The fixed-point free action of $C_3$ by tensor autoequivalences is unique, hence $\mathcal{C}_A$ is equivalent to $\mathrm{Rep}(C_{19}\rtimes C_3)$ as a fusion category.  For any non-symmetric braiding on $\mathrm{Rep}(C_{19}\rtimes C_3)$, the twists on the simple objects of dimension $3$ are distinct primitive $19$th roots of unity, hence, such a braided category has no fixed-point free tensor autoequivalences of order $2$.

\par Lastly, consider the case $(\mathcal{C}_A)_B$ has no nontrivial invertible objects.  Let $d_1,d_2,d_3$ be the dimensions of the three classes of six nontrivial simple objects with full twists $\theta_1,\theta_2,\theta_3$.  Note that if any of the normalized twists corresponding to $\theta_1,\theta_2,\theta_3$ are related by square Galois conjugacy, then there are at most $3$ distinct Frobenius-Perron dimensions among $1,d_1,d_2,d_3$.  This would imply the existence of nontrivial invertible objects, against our assumptions, by \cite[Lemma 4.2]{schopieray2023fixedpointfree}.  Therefore, $\theta_1,\theta_2,\theta_3$ are at most $24$th roots of unity, i.e.\ $\dim(\mathcal{C})=2^a\cdot3^b$ for some $a,b\in\mathbb{Z}_{\geq0}$ and therefore $\mathcal{C}$ is solvable \cite[Theorem 9.15.9]{tcat} in the sense of \cite[Definition 1.2]{solvable} and moreover weakly group-theoretical in the sense of \cite[Definition 1.1]{solvable}.  But this would imply $(\mathcal{C}_A)_B$ is pointed \cite[Theorem 1.1]{MR3770935} as $\mathcal{D}$ was maximal Tannakian, again against our assumptions.


\subsubsection{Tannakian subcategory has maximal rank $2$}

\begin{figure}[H]
\centering
\begin{align*}
\begin{array}{|c|c|c|c|c|}
\hline \mathcal{C} & \mathrm{FPdim}(\mathcal{C}) & \mathrm{FPdims} & C_\mathcal{C}(\mathcal{C}) & \# \\\hline\hline
\mathrm{Rep}(C_2)\boxtimes\mathcal{C}(C_3,q) & 6 & 1,1,1,1,1,1 & \mathrm{Rep}(C_2) & 2\\
\mathrm{Rep}(C_2,\nu)\boxtimes\mathcal{I}_q & 8 & 1,1,1,1,\sqrt{2},\sqrt{2} & \mathrm{Rep}(C_2,\nu) & 16 \\
\mathcal{I}_2 & 8 & 1,1,1,1,\sqrt{2},\sqrt{2} & \mathrm{Rep}(C_2) & 8\\
\mathcal{C}(C_2,q)\boxtimes\mathrm{Rep}(S_3)^\omega & 12 & 1,1,1,1,2,2 & \mathrm{Rep}(C_2) & 4\\
\mathrm{sVec}\boxtimes\mathrm{Rep}(S_3)^\omega & 12 & 1,1,1,1,2,2 & \mathrm{Rep}(C_2^2,\nu) & 2\\
\mathrm{Rep}(\mathrm{Dic}_3,\nu)^\omega & 12 & 1,1,1,1,2,2 & \mathrm{Rep}(C_2) & 3 \\
\mathcal{C}(C_6,q_{x,\omega})^{C_2} & 12 & 1,1,1,1,2,2 & \mathrm{Rep}(C_2) & 6\\
\mathrm{Rep}(C_3\rtimes S_3)^\omega & 18 & 1,1,2,2,2,2 & \mathrm{Rep}(C_2) & 1\\
\mathcal{C}(A_1,5,q)_\mathrm{ad}\boxtimes\mathrm{Rep}(S_3)^\omega & \frac{5}{2}\csc^2(\frac{\pi}{5}) & 1,1,[3]_5,[3]_5,2,2[3]_5 & \mathrm{Rep}(C_2) & 8 \\
\mathrm{Rep}(C_2)\boxtimes\mathcal{C}(A_1,7,q)_\mathrm{ad} & \frac{7}{2}\csc^2(\frac{\pi}{2}) & 1,1,[3]_7,[3]_7,[5]_7,[5]_7 & \mathrm{Rep}(C_2) & 3\\
\hline
\end{array}
\end{align*}
    \caption{The $53$ braided equivalence classes of premodular fusion categories of rank $6$ with Tannakian subcategory $\mathcal{D}$ of maximal rank $2$}%
    \label{fig:rank23456}%
\end{figure}

\par Assume now that $\mathcal{D}\simeq\mathrm{Rep}(C_2)\subset\mathcal{C}$ is a Tannakian subcategory of maximal rank.  If $\mathcal{D}\not\subset C_\mathcal{C}(\mathcal{C})$, then $\mathcal{D}\cap C_\mathcal{C}(\mathcal{C})=\mathrm{Vec}$ and thus $\mathcal{C}$ is supermodular since $C_\mathcal{C}(\mathcal{C})$ must contain a Tannakian subcategory of dimension at least one-half its own dimension \cite[Corollary 9.9.32]{tcat}.  These categories have been classified \cite[Section 3.1]{MR4109138} and in particular, only the products $\mathrm{sVec}\boxtimes\mathcal{I}_q$ contain a Tannakian subcategory of rank 2.  Otherwise $\mathcal{D}\subset C_\mathcal{C}(\mathcal{C})$ and therefore $\mathcal{C}_A$ is a braided fusion category where $A$ is the regular algebra of $\mathcal{D}\simeq\mathrm{Rep}(C_2)\subset C_\mathcal{C}(\mathcal{C})$.

\par First consider the case $\mathcal{D}\subsetneq C_\mathcal{C}(\mathcal{C})$ is a proper fusion subcategory.  Since $C_\mathcal{C}(\mathcal{C})$ is symmetrically braided with a maximal Tannakian subcategory of dimension $2$, then $C_\mathcal{C}(\mathcal{C})$ must be pointed of rank $4$ with exactly two invertible objects of trivial twist and two of twist $-1$.  The nontrivial invertible objects in the center must centralize the simple objects not contained in the center, and so they must permute them by the balancing equation \cite[Proposition 8.13.8]{tcat}.  Thus $\mathcal{C}$ is a braided generalized near-group fusion category.  We must have $\dim(\mathcal{C})\in\mathbb{Z}$ since there are no irrational examples with four invertible objects and rank $6$ \cite{MR4658217}.  The dimensions of the noninvertible simple objects are thus $\sqrt{2}$ or $2$ since no rank $6$ pointed fusion category has a symmetric center of rank $4$.  Therefore the dimensions of the simple objects of $\mathcal{C}$ are $1,1,1,1,\sqrt{2},\sqrt{2}$ or $1,1,1,1,2,2$.  In the former case, $\mathcal{C}$ is described by \cite[Theorem 5.5]{MR4195420}.  In the latter case, the fusion rules are either those of $\mathrm{Rep}(\mathrm{Dic}_3)$ or $\mathrm{Rep}(C_2\times S_3)$ and are described in detail in Example \ref{ex:12} which gives $2$ categories up to braided equivalence.

\par Now assume $\mathcal{D}=C_\mathcal{C}(\mathcal{C})$ so that $\mathcal{C}_A$ is a modular fusion category as $\mathcal{D}$ is of both maximal rank and dimension in this case.  There are three subcases depending on the number of fixed-points of the $C_2$-action by braided tensor autoequivalences on $\mathcal{O}(\mathcal{C})$ which we illustrate in Figure \ref{fig:rank2action}.

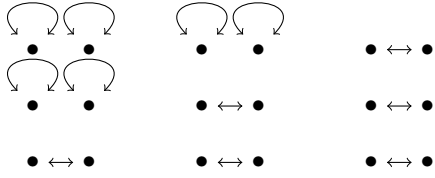
\begin{figure}[H]
\centering
\begin{align*}
\begin{tikzpicture}[scale=0.75]
\node (a) at (0,0) {$\bullet$};
\node (b) at (1,0) {$\bullet$};
\node (c) at (0,1) {$\bullet$};
\node (d) at (1,1) {$\bullet$};
\node (e) at (0,2) {$\bullet$};
\node (f) at (1,2) {$\bullet$};
\draw[<->] (a) -- (b);
\path[<->,out=45,in=135,looseness=5] (c) edge (c);
\path[<->,out=45,in=135,looseness=5] (d) edge (d);
\path[<->,out=45,in=135,looseness=5] (e) edge (e);
\path[<->,out=45,in=135,looseness=5] (f) edge (f);
\node (a1) at (3,0) {$\bullet$};
\node (b1) at (4,0) {$\bullet$};
\node (c1) at (3,1) {$\bullet$};
\node (d1) at (4,1) {$\bullet$};
\node (e1) at (3,2) {$\bullet$};
\node (f1) at (4,2) {$\bullet$};
\draw[<->] (a1) -- (b1);
\draw[<->] (c1) -- (d1);
\path[<->,out=45,in=135,looseness=5] (e1) edge (e1);
\path[<->,out=45,in=135,looseness=5] (f1) edge (f1);
\node (a2) at (6,0) {$\bullet$};
\node (b2) at (7,0) {$\bullet$};
\node (c2) at (6,1) {$\bullet$};
\node (d2) at (7,1) {$\bullet$};
\node (e2) at (6,2) {$\bullet$};
\node (f2) at (7,2) {$\bullet$};
\draw[<->] (a2) -- (b2);
\draw[<->] (c2) -- (d2);
\draw[<->] (e2) -- (f2);
\end{tikzpicture}
\end{align*}
    \caption{An illustration of the three cases of the $C_2$-action on $\mathcal{O}(\mathcal{C})$}%
    \label{fig:rank2action}%
\end{figure}

\paragraph{$C_2$-action has four fixed-points}  In this case $\mathcal{C}_A$ is a modular fusion category of rank $9$ and inversely, the $C_2$-action on $\mathcal{C}_A$ by braided tensor autoequivalences is fixed-point free.  Therefore the action is the duality autoequivalence \cite[Corollary 3.2]{schopieray2023fixedpointfree}.  As a result, $\mathcal{C}_A$ has odd dimension and is moreover pointed of rank $9$ (see \cite{MR4516198}, for example).  Therefore $\mathcal{C}$ has two invertible objects, and four simple objects of dimension $2$; all twists are $9$th roots of unity depending on the braiding and group structure of $\mathcal{C}_A$.  If any twist is trivial, then $\mathcal{C}_A$ has a nontrivial Tannakian subcategory of rank $3$, against our assumptions as then $\mathcal{C}$ would have a Tannakian subcategory braided equivalent to $\mathrm{Rep}(S_3)$.  Therefore the only option is that $\mathcal{C}\simeq(\mathcal{C}(C_3,q)^{\boxtimes2})^{C_2}$ is a braided equivalence where $C_2$ acts by duality; the braided categories $\mathcal{C}(C_3,q)^{\boxtimes2}$ for both nondegenerate quadratic forms $q$ are braided equivalent so this produces only one example. the adjoint subcategory of the categories $\mathcal{E}(q,\pm)$ from \cite[Example 5.3(b)]{MR2587410} with $q$ elliptic. 

\paragraph{$C_2$-action has two fixed-points} In this case $\mathcal{C}_A$ is modular of rank $6$, with two pairs of two simple objects with the same dimensions and same twists.  The only modular data of rank $6$ satisfying this constraint are pointed, or a product of a pointed modular fusion category of rank $3$ and $\mathcal{C}(A_1,5,q)_\mathrm{ad}$ with $q^2$ a primitive fifth root of unity \cite[Appendix E.5]{ng2023classification}.  Thus the dimensions of simple objects of $\mathcal{C}$ are $1,1,1,1,2,2$ or $1,1,2,(1/2)(1+\sqrt{5}),(1/2)(1+\sqrt{5}),1+\sqrt{5}$.  The former case was described in detail in Example \ref{ex:12}.  In the latter case, $C_2$ acts trivially by braided autoequivalences on the $\mathcal{C}(A_1,5,q)_\mathrm{ad}$ factor, hence $\mathcal{C}\simeq\mathrm{Rep}(S_3)^\omega\boxtimes\mathcal{C}(A_1,5,q)_\mathrm{ad}$ where $q^2$ is a primitive $5$th root of unity and $\omega$ a primitive third root of unity.

\paragraph{$C_2$-action is fixed-point free}  In this case $\mathcal{C}_A$ is a modular fusion category of rank $3$, which is either pointed, Ising, or a Galois conjugate of $\mathcal{C}(\mathfrak{sl}_2,7,q)_\mathrm{ad}$ where $q^2$ is a primitive $7$th root of unity (see Figure \ref{fig:catranklessthan4}).  There are no nontrivial braided autoequivalences of $\mathcal{C}(\mathfrak{sl}_2,7,q)_\mathrm{ad}$ \cite[Corollary 3.2]{MR4401829} so $\mathcal{C}\simeq\mathrm{Rep}(C_2)\boxtimes\mathcal{C}(\mathfrak{sl}_2,7,q)_\mathrm{ad}$ is a braided equivalence in the latter case; the same argument holds in the Ising case.  Otherwise, $\mathcal{C}$ is pointed of rank $6$ and the $C_2$-action on $\mathcal{O}(\mathcal{C}_A)$ is trivial, hence $\mathcal{C}\simeq\mathrm{Rep}(C_2)\boxtimes\mathcal{C}(C_3,q)$ is a braided equivalence for either nondegenerate quadratic form on $C_3$.


\subsubsection{Tannakian subcategory has maximal rank $1$}\label{sec6mod}

\par If $\mathcal{C}$ is a premodular fusion category of rank $6$ and there does not exist a Tannakian fusion subcategory of rank larger than $1$, then $C_\mathcal{C}(\mathcal{C})$ is braided equivalent to $\mathrm{Vec}$ or $\mathrm{sVec}$.  In the former case, $\mathcal{C}$ is a modular fusion category of rank $6$ whose modular data are listed in their entirety in \cite[Appendex E.5]{ng2023classification} (see also \cite[Tables 3-4]{MR4630478}).  The premodular fusion categories with $C_\mathcal{C}(\mathcal{C})\simeq\mathrm{sVec}$ a braided equivalence were classified in \cite{MR4109138}, and all but one family have a Tannakian subcategory of maximal rank $1$.  The braided equivalence classes of premodular categories for all of the above fusion rules are characterized as being products of smaller premodular categories or by \cite[Theorem A.3]{MR4486913} with three modular exceptions.  For two of these, indicated in Figure \ref{fig:nondegen} the authors know of no characterization of premodular categories with these fusion rules.  The last family is characterized as pair of inequivalent minimal modular extensions for each of the two inequivalent braided fusion categories $\mathrm{Rep}(D_5)^\mu$ from Figure \ref{fig:catrank4}.

\begin{figure}[H]
\centering
\begin{align*}
\begin{array}{|c|c|c|c|c|}
\hline \mathcal{C} & \mathrm{FPdim}(\mathcal{C}) & \mathrm{FPdims} &  \# \\\hline\hline
\mathcal{C}(C_2,q_1)\boxtimes\mathcal{C}(C_3,q_2) & 6 & 1,1,1,1,1,1  & 4 \\
\mathcal{C}(C_2,q_1)\boxtimes\mathcal{I}_{q_2} & 8 & 1,1,1,1,\sqrt{2},\sqrt{2} & 16 \\
\mathcal{C}(C_3,q_2)\boxtimes\mathcal{C}(A_1,5,q)_\mathrm{ad} & \frac{15}{4}\csc^2(\frac{\pi}{5}) & 1,1,1,[3]_5,[3]_5,[3]_5  & 8 \\
\mathcal{C}(A_1,5,q_1)_\mathrm{ad}\boxtimes\mathcal{I}_{q_2} & 5\csc^2(\frac{\pi}{5}) & 1,1,[3]_5,[3]_5, & 32 \\
 &  & [3]_5\sqrt{2},[3]_5\sqrt{2} &  \\
\mathcal{C}(C_2,q_1)\boxtimes\mathcal{C}(A_1,7,q_2)_\mathrm{ad} & \frac{7}{2}\csc^2(\frac{\pi}{7}) & 1,1,[3]_7,[3]_7,[5]_7,[5]_7  & 12 \\
\mathcal{C}(A_1,5,q_1)_\mathrm{ad}\boxtimes\mathcal{C}(A_1,7,q_2)_\mathrm{ad} & \frac{35\csc^2(\frac{\pi}{5})}{2(5-\sqrt{5})} & 1,[3]_7,[5]_7, &24 \\
& & [3]_5,[3]_5[3]_7,[3]_5[5]_7 & \\
\mathcal{C}(G_2,21,q) & \frac{21}{2}(5+\sqrt{21}) & 1,\frac{3+\sqrt{21}}{2},\frac{3+\sqrt{21}}{2},  & 12^\dagger \\
 & &\frac{3+\sqrt{21}}{2},\frac{5+\sqrt{21}}{2},\frac{7+\sqrt{21}}{2} &  \\
(\mathrm{Rep}(D_5)^\mu)^\gamma & 20 & 1,1,2,2,\sqrt{5},\sqrt{5}  & 4 \\
\mathcal{C}(A_1,13,q)_\mathrm{ad} & \frac{13}{4}\csc^2(\frac{\pi}{13}) & 1,[3]_{13},[5]_{13},  & 12 \\
& & [7]_{13},[9]_{13},[11]_{13} &  \\
\mathcal{C}(B_2,9,q)_\mathrm{ad} & 9u_1^2 &  1,u_1,u_1,u_1,u_1u_2,u_1^2u_2^{-1}  & 3^\dagger \\\hline
\hline
\mathrm{sVec}\boxtimes\mathcal{C}(C_3,q) & 6 & 1,1,1,1,1,1 & 2\\
\mathrm{sVec}\boxtimes\mathcal{C}(A_1,7,q)_\mathrm{ad} & \frac{7}{2}\csc^2(\frac{\pi}{7}) & 1,1,[3]_7,[3]_7,[5]_7,[5]_7 & 6 \\
\mathcal{C}(A_1,12,q)_\mathrm{ad} & 3\csc^2(\frac{\pi}{12}) & 1,1,[3]_{12},[5]_{12}, & 2  \\
& & [7]_{12},[9]_{12} & 
\\\hline
\end{array}
\end{align*}
    \caption{The $137$ braided equivalence classes of premodular fusion categories of rank $6$ whose Tannakian subcategory of maximal rank is $\mathrm{Vec}$, separated by modular (above) and supermodular (below).  We abbreviate $u_1=1-\zeta_9^4-\zeta_9^5$ and $u_2=\zeta_9-\zeta_9^2-\zeta_9^5$. \\[2mm] $^\dagger$ The authors are unaware of any proof in the literature that these are the only equivalence classes of premodular fusion categories with these fusion rules.}%
    \label{fig:nondegen}%
\end{figure}


\bibliographystyle{plain}
\bibliography{bib}

\end{document}